\documentclass[a4paper,11pt]{article}
\usepackage[latin1]{inputenc}
\usepackage[english]{babel}
\usepackage{amsmath}
\usepackage{amsfonts}
\usepackage{amssymb}
\usepackage{epsfig}
\usepackage{amsopn}
\usepackage{amsthm}
\usepackage{color}
\usepackage{graphicx}
\usepackage{enumerate}
\newtheorem{theorem}{Theorem}[section]

\newtheorem{lemma}[theorem]{Lemma}
\newtheorem{proposition}[theorem]{Proposition}
\newtheorem{definition}[theorem]{Definition}

\newtheorem*{theorem*}{Theorem}
\newtheorem*{lemma*}{Lemma}
\newtheorem*{remark*}{Remark}
\newtheorem*{definition*}{Definition}
\newtheorem*{proposition*}{Proposition}
\newtheorem*{corollary*}{Corollary}
\numberwithin{equation}{section}
\newcommand{\real}{\mathbb{R}}

\let\ced=\c         % cedilla
\def\a{\alpha}
\def\b{\beta}

\def\e{\varepsilon}        % Also, \varepsilon
\def\ca{{\cal A}}

\newcommand{\dx}{\,{\rm d}x}

\newcommand{\ds}{\,{\rm d}s}

\def\qed{\,\unskip\kern 6pt \penalty 500
\raise -2pt\hbox{\vrule \vbox to8pt{\hrule width 6pt
\vfill\hrule}\vrule}\par}
\definecolor{darkblue}{rgb}{0.05, .05, .65}
\definecolor{darkgreen}{rgb}{0.1, .65, .1}
\definecolor{darkred}{rgb}{0.8,0,0}
\newcommand{\beqn}{\begin{equation}}
\newcommand{\eeqn}{\end{equation}}
\newcommand{\bear}{\begin{eqnarray}}
\newcommand{\eear}{\end{eqnarray}}
\newcommand{\bean}{\begin{eqnarray*}}
\newcommand{\eean}{\end{eqnarray*}}
\begin{document}

\title{\huge \bf Existence and uniqueness of very singular solutions for a
fast diffusion equation with gradient absorption\footnote{Partially supported by Laboratoire Europ\'een Associ\'e CNRS Franco-Roumain MathMode Math\'ematiques \& Mod\'elisation}}

\author{
\Large Razvan Gabriel Iagar\,\footnote{Institut de Math\'ematiques
de Toulouse, CNRS UMR~5219, Universit\'e de Toulouse, F--31062
Toulouse Cedex 9, France.},\footnote{Institute of Mathematics of the
Romanian Academy, P.O. Box 1-764, RO-014700, Bucharest, Romania,
\textit{e-mail:} razvan.iagar@imar.ro.}
\\[4pt] \Large Philippe Lauren\c cot\,\footnote{Institut de
Math\'ematiques de Toulouse, CNRS UMR~5219, Universit\'e de
Toulouse, F--31062 Toulouse Cedex 9, France. \textit{e-mail:}
Philippe.Laurencot@math.univ-toulouse.fr}\\ [4pt] }
\date{}
\maketitle

%%%%%%%%%%%%%%%%%%%%%%%%%%%%%%%%%%%%%%%%%%%%%%%%%%%%%%%
\begin{abstract}
Existence and uniqueness of radially symmetric self-similar very singular solutions are proved for the singular diffusion equation
with gradient absorption
\begin{equation*}
\partial_t u -\Delta_{p}u+|\nabla u|^q=0, \ \hbox{in} \
(0,\infty)\times\real^N,
\end{equation*}
where $2N/(N+1)<p<2$ and $p/2<q<p-N/(N+1)$, thereby extending previous results restricted to $q>1$. 
\end{abstract}
%%%%%%%%%%%%%%%%%%%%%%%%%%%%%%%%%%%%%%%%%%%%%%%%%%%%%%%

\vspace{2.0 cm}

%%%%%%%%%%%%%%%%%%%%%%%%%%%%%%%%%%%%%%%%%%%%%%%%%%%%%%%
\noindent {\bf AMS Subject Classification:} 35K67, 35K92, 34B40, 34C11, 35B33.
%%%%%%%%%%%%%%%%%%%%%%%%%%%%%%%%%%%%%%%%%%%%%%%%%%%%%%%

\medskip

%%%%%%%%%%%%%%%%%%%%%%%%%%%%%%%%%%%%%%%%%%%%%%%%%%%%%%%
\noindent {\bf Keywords:}  Very singular solution, singular diffusion, gradient absorption,
self-similar solutions, $p$-Laplacian, uniqueness.
%%%%%%%%%%%%%%%%%%%%%%%%%%%%%%%%%%%%%%%%%%%%%%%%%%%%%%%
\newpage

%%%%%%%%%%%%%%%%%%%%%%%%%%%%%%%%%%%%%%%%
%%%%%%%%%%%%%%%%%%%%%%%%%%%%%%%%%%%%%%%%
\section{Introduction}
%%%%%%%%%%%%%%%%%%%%%%%%%%%%%%%%%%%%%%%%
%%%%%%%%%%%%%%%%%%%%%%%%%%%%%%%%%%%%%%%%

The singular diffusion equation with gradient absorption
\begin{equation}\label{FDE}
\partial_t u-\Delta_{p}u+|\nabla u|^q=0, \quad (t,x)\in
Q_{\infty}:=(0,\infty)\times\real^N,
\end{equation}
with $p\in (1,2)$ and $q>0$ features a singular diffusion term and an absorption term depending solely on the gradient. Taken apart, these two terms lead to two completely different behaviors for large times so that, putting them together, a competition between them is expected. In fact, we study recently in \cite{IL11} qualitative properties and decay estimates for nonnegative and bounded solutions to \eqref{FDE} and identify ranges of exponents $p\in (1,2)$ and $q>0$ with different behaviors. In particular, given $p\in (2N/(N+1),2)$, if $q>q_*:=p-(N/(N+1))$, the diffusion term dominates for large times while it is the absorption term that dominates when $q \in (0,p-1)$ leading to finite time extinction. Finite time extinction also occurs for $q\in (p-1,p/2)$ but is expected to be of a different nature, some influence of the diffusion term persisting near the extinction time. Finally, when $q\in (p/2,q_*)$, the solutions to \eqref{FDE} with initial data decaying sufficiently rapidly at infinity decay to zero at a faster algebraic rate than the one that would result from the diffusion alone, a feature which reveals an interplay between diffusion and absorption. By analogy with the existing literature on related problems (see, e.g., \cite{BKaL04, EK88, Va93} and the references therein), the large time behavior in that case is expected to be described by a particular self-similar solution to \eqref{FDE} which is called a \emph{very singular solution}. Recall that a very singular solution to a partial differential equation is a solution $U$ in $Q_\infty$ (in a weak or classical sense) such that 
\begin{equation}\label{cond2VSS}
\lim\limits_{t\to 0} \sup_{|x|>\e}\{ U(t,x) \} = 0,
\end{equation}
and
\begin{equation}\label{cond1VSS}
\lim\limits_{t\to 0}\int_{|x|<\e} U(t,x)\,\dx=\infty,
\end{equation}
for any $\e>0$ \cite{BPT}. In other words, the initial condition for $U$ is zero in $\real^ N\setminus\{0\}$ and it has a stronger singularity at the origin $x=0$ than the Dirac mass. Recall that a solution to a partial differential equation is usually referred to as a \emph{singular solution} if it satisfies \eqref{cond2VSS}, the main examples being the so-called \emph{fundamental solutions}, that is, solutions having a Dirac mass as initial condition, and the very singular solutions. The existence and non-existence of singular solutions have been thoroughly studied for diffusion equations with a zero-order absorption term and we refer to, e.g.,  \cite{BPT, KPV89, KVe88, Le96, Le97, PT86} for $\partial_t u - \Delta u^m + u^q =0$, $m>0$, and \cite{CQW03, KV92, PW88} for $\partial_t u - \Delta_p u + u^q =0$, $p>1$. The case of diffusion equations with an absorption term depending solely on the gradient has been investigated more recently, see, e.g., \cite{BL01,BKL04,QW01} for $\partial_t u - \Delta u^m + |\nabla u|^q =0$, $m\ge 1$, and \cite{Pe041} for $\partial_t u - \Delta_p u + |\nabla u|^q =0$, $p>2$. 

\medskip

In this paper we focus on the singular diffusion equation with gradient absorption \eqref{FDE} for the particular range of exponents $p$ and $q$ for which very singular solutions are likely to exist, namely:
\begin{equation}\label{exp}
p_c:=\frac{2N}{N+1}<p<2, \quad \frac{p}{2}<q<q_{*}=p-\frac{N}{N+1},
\end{equation}
as already mentioned. Owing to the homogeneity of \eqref{FDE}, we actually look for a (forward) self-similar and radially symmetric very singular solution $u$ to \eqref{FDE} of the form
\begin{equation*}
u(t,x)=t^{-\a}f(|x|t^{-\b}), \quad (t,x)\in Q_\infty\,,
\end{equation*}
for some profile $f$ and exponents $\a$ and $\b$ to be determined. Inserting this ansatz in \eqref{FDE} gives the values of $\a$ and $\b$
\begin{equation}\label{VSSS}
\a=\frac{p-q}{2q-p}, \qquad \b=\frac{q-p+1}{2q-p},
\end{equation}
and implies that the profile $f=f(r)$, $r=|x|t^{-\b}$, is a solution of the ordinary differential equation
\begin{equation}\label{ODE1}
(|f'|^{p-2}f')'(r)+\frac{N-1}{r}(|f'|^{p-2}f')(r)+\a f(r)+\b rf'(r)-|f'(r)|^q=0\,, \qquad r> 0,
\end{equation}
with $f'(0)=0$. The previous conditions \eqref{cond2VSS} and \eqref{cond1VSS} become
\begin{equation}
\lim\limits_{r\to\infty} r^{(p-q)/(q-p+1)}f(r)=0, \quad
\lim\limits_{t\to 0} t^{N\b-\a}\int_{r<\e
t^{-\b}}f(r)r^{N-1}\,dr=\infty,
\label{fantasio}
\end{equation}
the second one being satisfied if $f\in L^1(0,\infty;r^{N-1} dr)$ and $\a-N\b>0$, that is $q<q_{*}$. 

\medskip

We then prove the following result:
\begin{theorem}\label{th:main}
Assume that $p$ and $q$ satisfy \eqref{exp}. There is a unique nonnegative solution $f$ to \eqref{ODE1} which satisfies $f'(0)=0$ and 
\begin{equation}
\lim\limits_{r\to\infty} r^{(p-q)/(q-p+1)}\ f(r) = 0\,. 
\label{gaston}
\end{equation}
In fact, there is a positive constant $w^ *$ defined in \eqref{spip} below such that  
$$
\lim\limits_{r\to\infty} r^{p/(2-p)}\ f(r) = w^*\,.
$$
\end{theorem}

Let us first mention that Theorem~1.1 is given in \cite{Pe042} under the additional restriction that $q>1$ besides the constraints \eqref{exp} on $p$ and $q$. However, the proof given there does not seem to apply to the case $q\in (p/2,1]$ and some new ideas have to be introduced which actually work for the whole range defined in \eqref{exp}.  The proof given below is thus done for $p$ and $q$ satisfying \eqref{exp}. 

As usual, Theorem~\ref{th:main} is a consequence of a detailed study of the initial value problem associated to \eqref{ODE1}, that is, we consider the solution $f(.;a)$ to \eqref{ODE1} with initial data $f(0;a)=a>0$ and $f'(0;a)=0$. We first establish the well-posedness of this problem at the beginning of Section~\ref{sect.uniqVSS} together with some basic properties of its solutions in Section~\ref{sec:bpf}. In particular, we show that, if $a>0$ is such that $f(.;a)$ is positive in $(0,\infty)$, then $|f'(.;a)|$ is controlled by $f(.;a)^{2/p}$ as a consequence of sharp gradient estimates established in \cite{IL11}. This turns out to be a cornerstone of the proof as it allows us to transfer some asymptotic properties of $f(.;a)$ to $f'(.;a)$. We next prove in Section~\ref{sec:drcf} that a monotonicity property with respect to the initial condition $a$ is enjoyed by the solutions $f(.;a)$ under suitable conditions. Such a monotonicity property is also true for the singular diffusion equation with zero-order absorption studied in \cite{CQW03}. We then split the range of initial conditions $a\in (0,\infty)$ into three disjoint sets according to the behavior of the derivative of $r\mapsto r^ {p/(2-p)} f(r;a)$ and characterize these three sets in Sections~\ref{sec:ca}, \ref{sec:cc}, and~\ref{sec:cbu}. This study then guarantees the existence of at least one profile $f$ satisfying the properties listed in Theorem~\ref{th:main}. We finally borrow an argument from \cite{CQW03} to prove the uniqueness of this solution.

%%%%%%%%%%%%%%%%%%%%%%%%%%%%%%%%%%%%%%%%
%%%%%%%%%%%%%%%%%%%%%%%%%%%%%%%%%%%%%%%%
\section{Self-similar and radially symmetric very singular solutions}\label{sect.uniqVSS}
%%%%%%%%%%%%%%%%%%%%%%%%%%%%%%%%%%%%%%%%
%%%%%%%%%%%%%%%%%%%%%%%%%%%%%%%%%%%%%%%%

In order to construct a solution to \eqref{ODE1} satisfying \eqref{gaston}, we use a shooting method which leads us to the following initial value problem:
\begin{equation}\label{IVPa}
\left\{
\begin{array}{l}
(|f'|^{p-2}f')'(r)+\displaystyle{\frac{N-1}{r}} (|f'|^{p-2}f')(r)+\a f(r)+\b rf'(r)-|f'(r)|^q=0 , \\ 
 \\
f(0)=a, \ f'(0)=0,
\end{array}\right.
\end{equation}
where $a>0$ is an arbitrary positive real number and the condition $f'(0)=0$ follows from the radial symmetry and expected smoothness of very singular solutions. Introducing $F:=-|f'|^{p-2} f'$, we observe that \eqref{IVPa} also reads
\begin{equation}\label{phl1}
\left\{
\begin{array}{l}
f'(r) = - |F(r)|^{(2-p)/(p-1)} F(r) ,\\
 \\
F'(r)+\displaystyle{\frac{N-1}{r}} F(r) = \a f(r) - \b r |F(r)|^{(2-p)/(p-1)} F(r)  - |F(r)|^{q/(p-1)} ,\\ 
 \\
f(0)=a, \ g(0)=0.
\end{array}\right.
\end{equation}
Since $q>p-1$ and $(2-p)/(p-1)>0$, the right-hand side of \eqref{phl1} is locally Lipschitz continuous. The term involving  $(N-1)/r$ being handled as usual, there is a unique maximal $C^1$-smooth solution $(f(.;a),F(.;a))$ to \eqref{phl1}. Owing to the continuity of $f(.;a)$ and the positivity of $a$, $f(.;a)$ is clearly positive in a right-neighborhood of $r=0$ and we define
\begin{equation}
R(a) := \inf{\left\{ r\ge 0\ :\ f(r;a)=0 \right\}} > 0 .
\label{phl2}
\end{equation}
In the sequel, where there is no risk
of confusion, we will omit $a$ from the notation and let
$f=f(.;a)$.

%%%%%%%%%%%%%%%%%%%%%%%%%%%%%%%%%%%%%%%%
%%%%%%%%%%%%%%%%%%%%%%%%%%%%%%%%%%%%%%%%
\subsection{Basic properties of $f(.;a)$}\label{sec:bpf}
%%%%%%%%%%%%%%%%%%%%%%%%%%%%%%%%%%%%%%%%
%%%%%%%%%%%%%%%%%%%%%%%%%%%%%%%%%%%%%%%%

We first prove some basic properties of the profile $f=f(.;a)$.

%%%%%%%%%%%%%%%%%%%%%%%%%%%%%%%%%%%%%%%%
\begin{lemma}\label{lemma.basic}
Let $a>0$. We have $f'(r;a)<0$ for any $r\in(0,R(a))$ and $|f'(r;a)|\leq (\a
a)^{1/q}$ in $(0,R(a))$. Moreover, if $R(a)=\infty$, then
$\lim\limits_{r\to\infty}f(r;a)=\lim\limits_{r\to\infty}f'(r;a)=0$.
\end{lemma}
%%%%%%%%%%%%%%%%%%%%%%%%%%%%%%%%%%%%%%%%

\begin{proof}
It readily follows from \eqref{phl1} that $F'(0)=\a a/N>0$ so that there is 
$\delta>0$ such that $f'(r)<0$ for $r\in (0,\delta)$. Introducing $r_0:=\inf{\{ r\in (0,R(a))\ :\ f'(r)=0\}}$, 
we assume for contradiction that $r_0<R(a)$. Then, on the one hand, $F(r_0)=f'(r_0)=0$ and we deduce from \eqref{phl1} that $F'(r_0)=\a f(r_0)>0$. On the other hand, $F(r)\ge 0=F(r_0)$ for $r\in (0,r_0)$ which implies that $F'(r_0)\le 0$, whence a contradiction. Consequently, $r_0\ge R(a)$ and $f'<0$ in $(0,R(a))$. Consider now $R\in (0,R(a))$ and let $r_m$ be a point of minimum of $f'$ in $[0,R]$. Clearly, $r_m>0$ and either $r_m\in (0,R)$ and $f''(r_m)= 0$ or $r_m=R$ and $f''(r_m)\le 0$. In both cases, it follows from \eqref{IVPa} that $|f'(r_m)|^q - \a f(r_m)\le 0$, whence 
$$
|f'(r)|^q \le |f'(r_m)|^q \le \a f(r_m) \le \a f(0) =\a a\,, \quad r\in [0,R]\,.
$$

Assume now that $R(a)=\infty$. Since $f$ is decreasing, there exists
$l:=\lim\limits_{r\to\infty}f(r)\geq0$. Defining the following energy:
\begin{equation*}
E(r):=\frac{p-1}{p}|f'(r)|^p+\frac{\a}{2}f(r)^2, \quad r\ge 0,
\end{equation*}
it follows from \eqref{IVPa} that
\begin{eqnarray}
E'(r) & = & (p-1)(|f'|^{p-2}f'f'')(r)+\a (ff')(r) \nonumber\\
& = & -\frac{N-1}{r}|f'(r)|^p-\b r|f'(r)|^2-|f'(r)|^{q+1}<0. \label{phl3}
\end{eqnarray}
Hence $E$ is decreasing and positive in $(0,\infty)$, and has thus a limit as $r\to\infty$. This property ensures that $f'$ has a limit as $r\to\infty$ while \eqref{phl3} implies that $f'\in L^{q+1}(0,\infty)$. Combining these two facts entails that $f'(r)\longrightarrow 0$ as $r\to\infty$ and so does $F$. It then follows from \eqref{phl2} that $F'(r)\longrightarrow \a l$ as $r\to\infty$. Assume for contradiction that $l>0$. Then $F'(r)\ge \a l/2$ for $r$ large enough whence $F(r)\ge \a l r/4$ for $r$ large enough. Therefore $F(r)\longrightarrow \infty$ as $r\to\infty$ and a contradiction. We have thus shown that $l=0$.
\end{proof}

Another useful result is the following expansion of $f(.;a)$ near the origin:

%%%%%%%%%%%%%%%%%%%%%%%%%%%%%%%%%%%%%%%%
\begin{lemma}\label{lemma.expansion}
Let $a>0$. There exist three positive constants $C_1$, $C_2$, and $C_3$ such
that, as $r\to 0$, 
\begin{equation}\label{expansion}
\begin{split}
f(r;a)& =a - C_1\left(\frac{a\a}{N}\right)^{1/(p-1)}r^{p/(p-1)}+C_2\left(\frac{a\a}{N}\right)^{(q-p+2)/(p-1)}r^{(p+q)/(p-1)}\\
& + C_3\left(\frac{a\a}{N}\right)^{(3-p)/(p-1)}r^{2p/(p-1)}+o(r^{2p/(p-1)}),
\end{split}
\end{equation}
and
\begin{equation}\label{phl4}
\begin{split}
\partial_a f(r;a) & =1- \frac{C_1}{(p-1)a}\left(\frac{a \a}{N}\right)^{1/(p-1)} r^{p/(p-1)} \\
& + \frac{(2+q-p) C_2}{(p-1) a} \left(\frac{a\a}{N}\right)^{(q-p+2)/(p-1)} r^{(p+q)/(p-1)}+o(r^{(p+q)/(p-1)}).
\end{split}
\end{equation}
\end{lemma}
%%%%%%%%%%%%%%%%%%%%%%%%%%%%%%%%%%%%%%%%

\begin{proof}
We start with the fact that $(|f'|^{p-2}f')'(0)=-(a\a/N)$, so that, as $r\to 0$, 
\begin{equation}
f'(r)=-\left( \frac{a\a}{N} \right)^{1/(p-1)} r^{1/(p-1)} + o(r^{1/(p-1)}).
\label{phl5}
\end{equation}
After integration, we obtain
\begin{equation}
f(r) = a-\frac{p-1}{p}\left(\frac{a\a}{N}\right)^{1/(p-1)}r^{p/(p-1)}+o(r^{p/(p-1)}).
\label{phl6}
\end{equation}
Since \eqref{IVPa} also reads
\begin{equation}
\frac{d}{dr} \left( r^{N-1} |f'(r)|^{p-2} f'(r) \right) = r^{N-1}\ \left( |f'(r)|^q - \b r f'(r) - \a f(r)\right)\,,
\label{phl7}
\end{equation}
we infer from \eqref{phl5} and \eqref{phl6} that, as $r\to 0$, 
$$
\frac{1}{r^{N-1}} \frac{d}{dr} \left( r^{N-1} |f'(r)|^{p-2} f'(r) \right) = -a \a + \left( \frac{a \a}{N} \right)^{q/(p-1)} r^{q/(p-1)} + o(r^{q/(p-1)}) .
$$
Integrating once, we obtain the following expansion for $f'$:
\begin{equation}
\begin{split}
f'(r) & = -\left( \frac{a\a}{N} \right)^{1/(p-1)} r^{1/(p-1)} \\
& + \frac{1}{q+N(p-1)} \left( \frac{a\a}{N} \right)^{(q-p+2)/(p-1)} r^{(q+1)/(p-1)} + o(r^{(q+1)/(p-1)}).
\end{split}
\label{phl8}
\end{equation}
Another integration gives that, as $r\to 0$, 
\begin{equation}
\begin{split}
f(r) & = a - \frac{p-1}{p} \left( \frac{a\a}{N} \right)^{1/(p-1)} r^{p/(p-1)} \\
& + \frac{p-1}{(p+q)(q+N(p-1))} \left( \frac{a\a}{N} \right)^{(q-p+2)/(p-1)} r^{(p+q)/(p-1)} + o(r^{(p+q)/(p-1)}).
\end{split}
\label{phl9}
\end{equation}

Inserting \eqref{phl8} and \eqref{phl9} in \eqref{phl7} then leads to \eqref{expansion} after lengthy but simple computations, 
with the constants
\begin{equation*}
C_1=\frac{p-1}{p}, \ C_2=\frac{p-1}{(p+q)(q+N(p-1))}, \
C_3=\frac{(p-1)q}{2p^2(p+N(p-1))(2q-p)}.
\end{equation*}

Concerning $\partial_a f$, we first observe that differentiating \eqref{phl7} with respect to $a$ gives
\begin{eqnarray}
(p-1) \frac{d}{dr} \left( r^{N-1} |f'(r)|^{p-2} \partial_a f'(r) \right) & = & q r^{N-1} |f'(r)|^{q-2} f'(r) \partial_a f'(r) \nonumber \\ 
& & - \b r^N \partial_a f'(r) - \a r^{N-1} \partial_a f(r)\,.
\label{phl10}
\end{eqnarray}
Setting $Y'(r):= q |f'(r)|^{q-p} f'(r) /(p-1)$ for $r\in [0,R(a))$ with $Y(0)=0$, we have
$$
\frac{d}{dr} \left( r^{N-1} |f'(r)|^{p-2} \partial_a f'(r) e^{-Y(r)} \right) = - \left[ \b r^N \partial_a f'(r) + \a r^{N-1} \partial_a f(r)\right]\ \frac{e^{-Y(r)}}{p-1}\,.
$$
We next use the properties $\partial_a f(0;a)=1$, $\partial_a f'(0;a)=0$, and $Y'(r)\to 0$ as $r\to 0$ (the latter being true since $q>p-1$) to conclude that 
$$
\frac{d}{dr} \left( r^{N-1} |f'(r)|^{p-2} \partial_a f'(r) e^{-Y(r)} \right) = - \frac{\a}{p-1} r^{N-1} + o(r^{N-1})
$$
as $r\to 0$. Integrating this identity, we infer from \eqref{expansion} that 
$$
\partial_a f'(r) = - \frac{\a}{N(p-1)} \left( \frac{a \a}{N} \right)^{(2-p)/(p-1)} r^{1/(p-1)} + o(r^{1/(p-1)}).
$$
Integrating once more with respect to $r$ gives an expansion of $\partial_a f(r)$ as $r\to 0$. Inserting these expansions in \eqref{phl10} and arguing as in the proof of \eqref{expansion} give \eqref{phl4} after some computations.
\end{proof}

We now prove that, if $a>0$ is such that $R(a)=\infty$, then $|f'(.;a)|$ is controlled by $f(.;a)^{2/p}$. To this end, we first check that $f(.;a)$ can be associated to a solution to \eqref{FDE} which turns out to be a viscosity solution in the sense of Definition~\ref{def:vs} below. This then allows us to apply the optimal gradient estimates obtained in \cite{IL11}. 

%%%%%%%%%%%%%%%%%%%%%%%%%%%%%%%%%%%%%%%%
\begin{lemma}\label{le:solvs}
Let $a>0$ be such that $R(a)=\infty$ and $\e>0$. Setting 
$$
U_\e(t,x) := (t+\e)^{-\a} f(|x| (t+\e)^{-\b} ; a)\,, \quad (t,x)\in [0,\infty)\times \real^N\,,
$$
the function $U_\e$ is a viscosity solution to \eqref{FDE} with initial condition $U_\e(0)$ defined by $U_\e(0,x) := \e^{-\a} f(|x| \e^{-\b};a)$ for $x\in\real^N$.
\end{lemma}
%%%%%%%%%%%%%%%%%%%%%%%%%%%%%%%%%%%%%%%%

\begin{proof}
Let us first observe that, owing to \eqref{phl5}-\eqref{phl7}, we have as $r\to 0$
\begin{equation}
(p-1) f''(r) = - \left( \frac{a \a}{N} \right)^{1/(p-1)} r^{(2-p)/(p-1)} + o(r^{(2-p)/(p-1)}).
\label{phl11}
\end{equation}
Since $p<2$ and $f$ clearly belongs to $C^2((0,\infty))$, we deduce from \eqref{phl11} that $f\in C^2([0,\infty))$ with $f''(0)=0$. Consequently, $U_\e$ belongs to $C^2([0,\infty)\times\real^N)$ and is clearly a classical solution to \eqref{FDE} in $(0,\infty)\times \real^N\setminus\{0\}$ where $\nabla U_\e$ does not vanish. Then $U_\e$ is a viscosity solution to \eqref{FDE} in $(0,\infty)\times \real^N\setminus\{0\}$ as it obviously satisfies Definition~\ref{def:vs} for all $(t_0,x_0)\in (0,\infty)\times \real^N\setminus\{0\}$. 

Consider now $t_0>0$ and $x_0=0$. We first prove that $U_\e$ is a viscosity subsolution. To this end, let $\psi\in\ca$ be such that $U_\e(t_0,0)=\psi(t_0,0)$ and $U_\e(t,x)<\psi(t,x)$ for any $(t,x)\in Q_{\infty}\setminus\{(t_0,0)\}$, the set $\ca$ of admissible comparison functions being defined in the Appendix. Since $U_\e$ and $\psi$ are both $C^1$-smooth, this property implies that $\nabla\psi(t_0,0)=\nabla U_\e(t_0,0)=0$ and $\partial_t\psi(t_0,0)=\partial_t U_\e(t_0,0)=-\a a (t_0+\e)^{\a+\b}<0$, so that Definition~\ref{def:vs} is satisfied.

We next prove that $U_\e$ is a viscosity supersolution, that is, $-U_\e$ is
a viscosity subsolution. Let $\psi\in\ca$ be such that $-U_\e(t_0,0)=\psi(t_0,0)$ and $-U_\e(t,x)<\psi(t,x)$
for any $(t,x)\in Q_{\infty}\setminus\{(t_0,0)\}$. Since $U_\e$ and $\psi$ are both $C^1$-smooth, we have
$\nabla\psi(t_0,0)=\nabla U_\e(t_0,0)=0$ and $\partial_t\psi(t_0,0)=-\partial_t U_\e(t_0,0)$. Therefore, for $x\in\real^N$,
$$
\psi(t_0,x)-\psi(t_0,0) \ge -U_\e(t_0,x) + U_\e(t_0,0) = - (t_0+\e)^{-\a} \left[ f(|x|(t_0+\e)^{-\b};a) - f(0;a) \right]\,,
$$
and it follows from \eqref{expansion} that, as $x\to 0$, 
$$
\psi(t_0,x)-\psi(t_0,0) \ge C_1 (t_0+\e)^{-\a} \left( \frac{a \a}{N} \right)^{1/(p-1)} \left( |x|(t_0+\e)^{-\b} \right)^{p/(p-1)} + o(|x|^{p/(p-1)})\,.
$$
However, since $\psi\in\ca$, there is $\xi\in\Xi$ such that, as $x\to 0$, $\xi(|x|)\ge \psi(t_0,x)-\psi(t_0,0)$. Combining the above two inequalities leads us to 
$$
\liminf_{r\to 0} \frac{\xi(r)}{r^{p/(p-1)}} > 0 \,,
$$
which contradicts \eqref{phl0}. This situation thus cannot occur and the proof is complete.
\end{proof}

As a consequence of Lemma~\ref{le:solvs} and since $p>p_c$, we may apply the gradient estimates proved for solutions to \eqref{FDE} in \cite[Theorems~1.2 and~1.3]{IL11}.

%%%%%%%%%%%%%%%%%%%%%%%%%%%%%%%%%%%%%%%%
\begin{proposition}\label{prop.gradient}
Let $a>0$ be such that $R(a)=\infty$. Then $f(.;a)$ satisfies
\begin{equation}\label{phl12}
\left| f'(r;a)\right| \le C_4\ f(r;a)^{2/p}\,, \quad r\ge 0\,,
\end{equation}
for some constant $C_4$ depending only on $N$,  $p$, $q$ and $a$.
\end{proposition}
%%%%%%%%%%%%%%%%%%%%%%%%%%%%%%%%%%%%%%%%

\begin{proof}
Fix $\e>0$. According to Lemma~\ref{le:solvs}, the function $U_\e$ defined by $U_\e(t,x) = (t+\e)^{-\a} f(|x| (t+\e)^{-\b} ; a)$ for $(t,x)\in [0,\infty)\times \real^N$ is a viscosity solution to \eqref{FDE} with a nonnegative initial condition which belongs to $W^{1,\infty}(\real^N)$ by Lemma~\ref{lemma.basic}. 

If $q\ge 1$, we infer from \cite[Theorem~1.2]{IL11} that there is a positive constant $C$ depending only on $N$ and $p$ such that 
$$
\left| \nabla U_\e^{-(2-p)/p}(t,x) \right| \le C\ t^{-1/p}\,, \quad (t,x)\in Q_\infty.
$$
In terms of $f(.;a)$, we obtain
$$
\left| f'(|x| (t+\e)^{-\b}) \right| \le C\ \left( \frac{t+\e}{t} \right)^{1/p}\ f(|x| (t+\e)^{-\b})^{2/p}\,, \quad (t,x)\in Q_\infty.
$$
Letting $\e\to 0$ and choosing $t=1$ give \eqref{phl12}.

If $q\in (p/2,1)$, we infer from \cite[Theorem~1.3]{IL11} that there is a positive constant $C$ depending only on $N$, $p$, and $q$ such that 
$$
\left| \nabla U_\e^{-(2-p)/p}(t,x) \right| \le C\ \left( \left\| U_\e(t/2) \right\|_\infty^{1/\a p} + t^{-1/p} \right)\,, \quad (t,x)\in Q_\infty.
$$
In terms of $f(.;a)$, we obtain
$$
\left| f'(|x| (t+\e)^{-\b}) \right| \le C\ (t+\e)^{1/p}\ \left( a^{1/\a p} (t+2\e)^{-1/p} + t^{-1/p} \right)\ f(|x| (t+\e)^{-\b})^{2/p}
$$
for $(t,x)\in Q_\infty$. Letting $\e\to 0$ and choosing $t=1$ give \eqref{phl12}.
\end{proof}

%%%%%%%%%%%%%%%%%%%%%%%%%%%%%%%%%%%%%%%%
%%%%%%%%%%%%%%%%%%%%%%%%%%%%%%%%%%%%%%%%
\subsection{Decay rates and monotonicity}\label{sec:drcf}
%%%%%%%%%%%%%%%%%%%%%%%%%%%%%%%%%%%%%%%%
%%%%%%%%%%%%%%%%%%%%%%%%%%%%%%%%%%%%%%%%

We now study the possible decay rates as $r\to\infty$ of the profile
$f(.;a)$ (when $R(a)=\infty$). Since we expect an algebraic decay, we make the following
ansatz:
\begin{equation*}
f(r;a)\sim C r^{1-\gamma}, \ f'(r;a)\sim C (1-\gamma) r^{-\gamma}, \quad
\hbox{as} \ r\to\infty,
\end{equation*}
for some $\gamma>1$. Inserting this ansatz in \eqref{IVPa}, we easily
find that there are only two possibilities: $\gamma=2/(2-p)$ and
$\gamma=1/(q-p+1)$. In the former case, $f(r;a)\sim C r^{-p/(2-p)}$ as $r\to \infty$,
while $f(r;a)\sim C r^{-(p-q)/(q-p+1)}$ as $r\to\infty$ in the latter. Observing that 
$(p-q)/(q-p+1)<p/(2-p)$ as $q>p/2$, the solutions to \eqref{IVPa} decaying with the first rate are called \emph{fast orbits} while those decaying with the second one are called \emph{slow orbits} \cite{BPT}. Complying with \eqref{cond2VSS} requires $f$ to be a fast orbit, and we now proceed to show the existence of such solutions to \eqref{IVPa}. 

To this end, following \cite{CQW03, Pe042}, we introduce the new unknown function
\begin{equation*}
w(r)=w(r;a) := r^{p/(2-p)} f(r;a), \quad (r,a)\in [0,R(a))\times (0,\infty),
\end{equation*}
which is a solution of the following differential equation: 
\begin{equation}\label{ODE2}
\begin{split}
& (p-1)r^2w''(r)+[N-1-2\mu(p-1)]rw'(r)+\mu(\mu-N)w(r)\\
& \quad +|rw'(r)-\mu w(r)|^{2-p}\left[(\a-\b\mu)w(r)+\b rw'(r)-r^{\eta}|rw'(r)-\mu w(r)|^q\right]=0
\end{split}
\end{equation}
for $r\in [0,R(a))$, where
$$
\mu:=\frac{p}{2-p}>N, \quad \eta:=-\frac{2q-p}{2-p}<0, \quad
\a-\b\mu=-\frac{1}{2-p}<0.
$$
We end the preliminary results with the following monotonicity property with
respect to the initial value $a$ which will be very useful in the uniqueness proof later on.
We define a linear differential operator $L_a$ by:
\begin{equation}\label{linearized}
\begin{split}
L_a(\varphi)(r)&:=(p-1)r^2\varphi''(r)+[N-1-2\mu(p-1)]r\varphi'(r)+\mu(\mu-N)\varphi(r)\\
&+(2-p)(|W|^{-p} W)(r) (r\varphi'(r)-\mu
\varphi(r))\left[(\a-\b\mu)w(r)+\b rw'(r)-r^{\eta}|W(r)|^{q}\right]\\
&+|W(r)|^{2-p}\left[(\a-\b\mu)\varphi(r)+\b
r\varphi'(r)-qr^{\eta} (|W|^{q-2} W)(r) (r\varphi'(r)-\mu
\varphi(r))\right],
\end{split}
\end{equation}
where $W(r):=rw'(r)-\mu w(r)$ and $w(r)=w(r;a)$ for $r\in [0,R(a))$ in \eqref{linearized}. Before stating the monotonicity lemma, we gather some properties of $L_a$ in the following two results.

%%%%%%%%%%%%%%%%%%%%%%%%%%%%%%%%%%%%%%%%
\begin{lemma}\label{lemma.monot2}
Given $a>0$, we have 
\begin{equation}
L_a(\partial_a w(.;a))=0 \;\;\mbox{ and }\;\; L_a(rw'(.;a))<0 \;\;\mbox{ in }\;\;  (0,R(a))\,.
\label{phl13}
\end{equation} 
Moreover, for $l\in (0,\mu a]$, there exists a small interval $(0,s_l(a))$ such that
\begin{eqnarray*}
\mu aw_a(r;a) & > & rw'(r;a) \;\; \hbox{for} \;\; r\in (0,s_{\mu a}(a))\,, \\
lw_a(r;a) & < & rw'(r;a) \;\; \hbox{for} \;\; l\in (0,\mu a) \;\;\mbox{ and }\;\; r\in(0,s_l(a)).
\end{eqnarray*}
\end{lemma}
%%%%%%%%%%%%%%%%%%%%%%%%%%%%%%%%%%%%%%%%

\begin{proof}
To simplify notations, we set $w_a:= \partial_a w(.;a)$ and $f_a:= \partial_a f(.;a)$. Differentiating the ordinary differential equation \eqref{ODE2} with respect to the parameter $a$, we readily obtain that $L_a(w_a)=0$ in $(0,R(a))$. Next, differentiating \eqref{ODE2} with respect to $r$ and multiplying the resulting identity by $r$ give, by a straightforward computation
$$
L_a(r w')=\eta r^{\eta}|rw'(r)-\mu w(r)|^{q-p+2}<0 \ \hbox{in} \ (0,R(a)).
$$
We now prove the last two assertions and consider $l\in (0,\mu a]$. Since $w_a(r)=r^{\mu}f_a(r)$ and
$rw'(r)=r^{\mu}(rf'(r)+\mu f(r))$, we have $lw_a(r)-r w'(r)= r^\mu \left( l f_a(r)-r f'(r)-\mu f(r) \right)$ and use Lemma~\ref{lemma.expansion} to identify the behavior of this function in a right-neighborhood of $r=0$. Indeed, we infer from Lemma~\ref{lemma.expansion}, \eqref{phl8}, and the definitions of $C_1$ and $C_2$ that, as $r\to 0$,
\begin{eqnarray*}
& & l f_a(r)-r f'(r)-\mu f(r) \\
& = & l - \mu a - C_1 \left(\frac{a \a}{N}\right)^{1/(p-1)} \left[ \frac{l}{(p-1) a} - \mu - \frac{p}{p-1}\right]\ r^{p/(p-1)} \\
\quad \quad & + & C_2 \left(\frac{a\a}{N}\right)^{(q-p+2)/(p-1)} \left[ \frac{(2+q-p) l}{(p-1) a} - \frac{(p+q)}{p-1} - \mu \right]\ r^{(p+q)/(p-1)} \\
\quad \quad & + & o(r^{(p+q)/(p-1)})\,. 
\end{eqnarray*}
On the one hand, if $l<\mu a$, we have $l f_a(r)-r f'(r)-\mu f(r)\sim l - \mu a <0$ as $r\to 0$, which implies that this function is negative in $(0,s_l(a))$ for some sufficiently small $s_l(a)>0$. On the other hand, if $l=\mu a$, we have
$$
l-\mu a = 0\,, \quad  \frac{l}{(p-1) a} - \mu - \frac{p}{p-1} = \frac{\mu}{(p-1)} - \mu - \frac{p}{p-1} = 0
$$
and
$$
\frac{(2+q-p) l}{(p-1) a} - \frac{(p+q)}{p-1} - \mu =  \frac{\mu(2+q-p)}{p-1} - \frac{p+q}{p-1} - \mu = \frac{(2q-p)(p-1)}{2-p}>0,
$$
hence $\mu a f_a(r)-rf'(r)-\mu f(r)>0$ for $r>0$ sufficiently small. 
\end{proof}

%%%%%%%%%%%%%%%%%%%%%%%%%%%%%%%%%%%%%%%%
\begin{lemma}[Comparison principle]\label{lemma.comparison}
Let $a>0$, $r_1\in (0,R(a))$, and $r_2\in(r_1,R(a))$ and assume that $w'(.;a)>0$ in
$[r_1,r_2]$. Then, for any function $h\in C^2(r_1,r_2)$ satisfying
$h(r_1)=h(r_2)=0$ and $L_a(h)\geq 0$ in $(r_1,r_2)$, we have $h\le 0$ in $(r_1,r_2)$.
\end{lemma}
%%%%%%%%%%%%%%%%%%%%%%%%%%%%%%%%%%%%%%%%

\begin{proof}
Lemma~\ref{lemma.comparison} is a variant of the classical maximum principle \cite[Corollary~3.2]{GT01} and is proved in
\cite[p.~48]{BNV}. The proof relies on the existence of a function $\varphi>0$ on $[r_1,r_2]$ such that $L_a(\varphi)\le 0$. Here this is satisfied by $\varphi(r)=rw'(r;a)$ in view of Lemma~\ref{lemma.monot2} and the assumption on $w'(.;a)$.
\end{proof}
We are now in a position to state and prove the monotonicity result
with respect to $a$.

%%%%%%%%%%%%%%%%%%%%%%%%%%%%%%%%%%%%%%%%
\begin{lemma}[Monotonicity lemma]\label{monotonicity}
Let $a>0$ and suppose that $w'(.;a)>0$ in a finite interval $(0,r_0)$ with $r_0<R(a)$.
Then we have $\mu a \partial_a w(r;a)>rw'(r;a)$ for $r\in (0,r_0)$ and $\partial_a w(.;a)>0$ on $[0,r_0]$.
\end{lemma}
%%%%%%%%%%%%%%%%%%%%%%%%%%%%%%%%%%%%%%%%

\begin{proof}
We adapt the proof of \cite[Lemma 2.2]{CQW03}. To simplify notations, we again set $w_a:= \partial_a w(.;a)$ and $f_a:= \partial_a f(.;a)$ and first show that $\mu a w_a(r)>rw'(r)$ for $r\in (0,r_0)$. Assume for contradiction that there exists some point $r_2\in (0,r_0)$ such that $\mu a w_a(r)-rw'(r)>0$ for $r\in (0,r_2)$ and $\mu a w_a(r_2)-r_2 w'(r_2)=0$ at $r=r_2$. Then, given $l\in (0,\mu a)$, we have $l w_a(r_2)-r_2 w'(r_2)<0$ and, since $l w_a(r) - r w'(r)<0$ for $r\in (0,s_l(a))$ by Lemma~\ref{lemma.monot2}, there is $l\in (0,\mu a)$ sufficiently close to $\mu a$ such that we have $l w_a(r_1)-r_1 w'(r_1) = l w_a(r_3) - r_3 w'(r_3)=0$ and  $l w_a(r) -r w'(r)>0$ for $r\in (r_1,r_3)$ for some $0<r_1<r_3<r_2$. We are then in a position to apply Lemma~\ref{lemma.comparison} to the function $h(r):=lw_a(r)-rw'(r)$ in the interval $[r_1,r_3]\subset(0,r_2)$, recalling that $L_a(h)=lL_a(w_a)-L_a(rw')=-L_a(rw')>0$ in $(r_1,r_3)$ by Lemma~\ref{lemma.monot2}. This implies that $lw_a(r)-rw'(r)\leq 0$ for $r\in (r_1,r_3)$ and a contradiction. Therefore 
\begin{equation}
\mu a w_a(r)>rw'(r) \;\;\mbox{ for }\;\; r\in (0,r_0),
\label{phl14}
\end{equation} 
and as a direct consequence, $w_a>0$ in $(0,r_0)$.

It remains to check that $w_a(r_0)>0$. Fix $s_0\in (0,r_0)$
and define $l_0:=s_0w'(s_0)/w_a(s_0)$ and $k_0:=(l_0w_a-rw')'(s_0)$. We first observe that, since $s_0\in (0,r_0)$, \eqref{phl14} guarantees that $l_0\in (0,\mu a)$. We next show that $k_0>0$. Indeed, if $k_0<0$, we have $l_0 w_a(r)-rw'(r)> 0$ in a left-neighborhood of $s_0$ as $l_0 w_a(s_0)-s_0w'(s_0)= 0$. Defining
$$
s_1 := \sup{\{ r\in (0,s_0)\ :\ l_0 w_a(r) - r w'(r) = 0 \}}\,,
$$
we have just established that $s_1<s_0$ while Lemma~\ref{lemma.monot2} guarantees that $s_1>0$ as $l_0<\mu a$. Therefore, setting $h(r):=l_0 w_a(r) - r w'(r)$ for $r\in [s_1,s_0]$, we have $h(s_1)=h(s_0)=0$ and $L_a(l_0 w_a - r w')>0$ in $(s_1,s_0)$. This implies that $h\le 0$ in $(s_1,s_0)$ by Lemma~\ref{lemma.comparison} and a contradiction. Next, if $k_0=0$,  it follows from Lemma~\ref{lemma.monot2} that $(p-1) s_0^2 (l_0 w_a - r w')''(s_0) = L_a(l_0 w_a - r w')(s_0)>0$ from which we deduce that $l_0 w_a(r)-rw'(r)> 0$ in a left-neighborhood of $s_0$. Arguing as in the previous case, we are again led to a contradiction and conclude that $k_0>0$.

Consider now the solution $\psi$ of $L_a(\psi)=0$ in $(0,R(a))$ with initial conditions $\psi(s_0)=0$ and $\psi'(s_0)=1$. We claim that $\psi$ cannot vanish in $(s_0,r_0]$. Indeed, since $L_a(w_a)=0$ with $w_a(s_0)>s_0 w'(s_0)/\mu a>0$ by Lemma~\ref{lemma.monot2} and \eqref{phl14}, $w_a$ and $\psi$ are linearly independent solutions to the same second-order ordinary differential equation and the oscillation theorem guarantees that there is at least one zero of $w_a$ between two zeros of $\psi$. Thus, if $\psi(s)=\psi(s_0)=0$ for some $s\in (s_0,r_0]$, the function $w_a$ has to vanish at least once in $(s_0,s)$ which contradicts \eqref{phl14}. Therefore, $\psi(r)>0$ for $r\in (s_0,r_0]$. 

Define now the function $\varphi:=l_0 w_a-k_0\psi$. Then, by construction and Lemma~\ref{lemma.monot2},
\begin{equation*}
L_a(\varphi-rw')>0 \ \hbox{in} \ (0,R(a)), \quad \varphi(s_0)- s_0 w'(s_0)=(\varphi-rw')'(s_0)=0,
\end{equation*}
so that $(p-1) s_0^2 (\varphi-rw')''(s_0) = L_a(\varphi-rw')(s_0)>0$. Consequently, $\varphi(r)-rw'(r)>0$ in a right-neighborhood of $s_0$. Assume for contradiction that there is $s_2\in (s_0,r_0]$ such that $\varphi(r)-rw'(r)>0$ for $r\in (s_0,s_2)$ and $\varphi(s_2)-s_2w'(s_2)=0$. As $L_a(\varphi-rw')>0$ in $(s_0,s_2)$, we deduce from Lemma~\ref{lemma.comparison} that $\varphi(r)-rw'(r)\le 0$ for $r\in (s_0,s_2)$ and a contradiction. Consequently, $\varphi(r)-r w'(r)>0$ for $r\in (s_0,r_0]$. In particular, 
\begin{equation*}
l_0w_a(r_0)=\varphi(r_0)+k_0\psi(r_0)>r_0 w'(r_0) + k_0\psi(r_0)>0,
\end{equation*}
which ends the proof.
\end{proof}

As in \cite{CQW03, Pe042}, we split the range of $a$ in three disjoint sets:
\begin{eqnarray*}
A & := & \left\{a>0\ :\ \mbox{ there exists}\;\; R_1(a)\in(0,R(a)), \ w'(R_1(a);a)=0 \right\}, \\
B & := & \left\{a>0\ :\ w'(.;a)>0 \ \hbox{in} \ (0,\infty), \ \lim\limits_{r\to\infty} w(r;a)<\infty  \right\}, \\
C & := & \left\{a>0\ :\ w'(.;a)>0 \ \hbox{in} \ (0,\infty), \ \lim\limits_{r\to\infty} w(r;a)=\infty  \right\}.
\end{eqnarray*}
Since $w'(r;a)\sim \mu a r^{\mu-1}$ as $r\to 0$, we have $w'(.;a)>0$ in a right-neighborhood of $r=0$ from which we deduce that the three sets are disjoint and $A\cup B\cup C=(0,\infty)$. According to the discussion at the beginning of Section~\ref{sec:drcf} and the definition of $w$, a fast orbit corresponds to $a\in B$ while a slow orbit starts from $a\in C$. Thus, the first step is to show that $B$ is non-empty which will follow from the fact that $A$ and $C$ are both open intervals. In a second step, we shall prove that $B$ reduces to a single point, thereby showing the existence and uniqueness of a radially symmetric self-similar very singular solution a stated in Theorem~\ref{th:main}. 

%%%%%%%%%%%%%%%%%%%%%%%%%%%%%%%%%%%%%%%%
%%%%%%%%%%%%%%%%%%%%%%%%%%%%%%%%%%%%%%%%
\subsection{Characterization of $A$}\label{sec:ca}
%%%%%%%%%%%%%%%%%%%%%%%%%%%%%%%%%%%%%%%%
%%%%%%%%%%%%%%%%%%%%%%%%%%%%%%%%%%%%%%%%

We show that, for all $a\in A$, the profile $f(.;a)$ crosses the $r$-axis, that is, $R(a)<\infty$. To this end, we introduce
\begin{equation}
w^*:=\left(\frac{\mu^{p-1}(\mu-N)}{\mu\b-\a}\right)^{1/(2-p)},
\label{spip}
\end{equation}
which is the constant solution of \eqref{ODE2} without the term $r^\eta |rw'(r) - \mu w(r)|^{2+q-p}$ which is expected to be negligible in the limit $r\to\infty$ as $\eta<0$.

%%%%%%%%%%%%%%%%%%%%%%%%%%%%%%%%%%%%%%%%
\begin{lemma}\label{lemma.caractA1}
Let $a>0$. Then, $a\in A$ if and only if there exists
$R_1(a)\in(0,R(a))$ such that $w'(.;a)>0$ in $(0,R_1(a))$,
$w'(.;a)<0$ in $(R_1(a),R(a))$ and $w''(R_1(a);a)<0$. Moreover, if $a\in
A$, then
\begin{equation}\label{sup.crit}
\sup\limits_{r\in(0,R(a))}w(r;a)<w^*.
\end{equation}
\end{lemma}
%%%%%%%%%%%%%%%%%%%%%%%%%%%%%%%%%%%%%%%%

\begin{proof}
We adapt the technique in \cite{CQW03}. Let $a\in A$ and denote the smallest positive zero of $w'$ in $(0,R(a))$ by $R_1(a)$, its existence being guaranteed as $a\in A$. Then, $w'>0$ in $(0,R_1(a))$ and $w''(R_1(a))\leq 0$. Assume for contradiction that $w''(R_1(a))=0$. Differentiating \eqref{ODE2} with respect to $r$ and taking $r=R_1(a)$ in the resulting identity give
\begin{equation*}
(p-1)R_1(a)^2 w'''(R_1(a)) = \eta R_1(a)^{\eta-1} (\mu w(R_1(a))^{2+q-p}<0.
\end{equation*}
Consequently, there exists some $\delta>0$ such that $w''>0$ in $(R_1(a)-\delta,R_1(a))$, which entails that $w'(r)<w'(R_1(a))=0$ for $r\in (R_1(a)-\delta,R_1(a))$ and a contradiction with the definition of $R_1(a)$. Hence, $w''(R_1(a))<0$.

We show next that $w'<0$ in $(R_1(a),R(a))$. Since $w'(R_1(a))=0$ and $w''(R_1(a))<0$, we clearly have $w'<0$ in a right-neighborhood of $R_1(a)$. Assume for contradiction that there is $R_2(a)\in(R_1(a),R(a))$ such that $w'<0$ in $(R_1(a),R_2(a))$ and $w'(R_2(a))=0$. Then $w''(R_2(a))\geq0$. We evaluate \eqref{ODE2} at $R_i(a)$, $i=1,2$, and find
\begin{equation*}
\begin{split} \mu(\mu-N)w(R_i(a))&+(\a-\b\mu)(\mu
w(R_i(a)))^{2-p}w(R_i(a))-R_i(a)^{\eta}(\mu
w(R_i(a)))^{q-p+2}\\&=-(p-1)R_i(a)^2w''(R_i(a)),
\end{split}
\end{equation*}
for $i=1,2$. Since $w''(R_1(a))<0\leq w''(R_2(a))$ we deduce
\begin{equation}\label{ineq.wstar}
\begin{split}
(\mu w(R_1(a)))^{2-p}&\left[\b\mu-\a+\mu R_1(a)^{\eta}(\mu
w(R_1(a)))^{q-1}\right]<\mu(\mu-N)\\&\leq(\mu
w(R_2(a)))^{2-p}\left[\b\mu-\a+\mu R_2(a)^{\eta}(\mu
w(R_2(a)))^{q-1}\right].
\end{split}
\end{equation}
But this is a contradiction, since $w(R_1(a))>w(R_2(a))$, $2-p>0$,
$2-p+q-1=q-p+1>0$, $\b\mu-\a>0$ and $R_1(a)^{\eta}>R_2(a)^{\eta}$.
Thus, $w'<0$ in $(R_1(a),R(a))$.

The converse statement being obviously true, it remains to show \eqref{sup.crit}. If $a\in A$, $w(.)$ has a single critical point $R_1(a)$ in $(0,R(a))$ and $w$ attains its maximum at this point. The inequality \eqref{sup.crit} then readily follows from the first inequality in \eqref{ineq.wstar}.
\end{proof}

The next step is to show that any profile $f(.;a)$ starting from $a\in A$ crosses the $r$-axis. The following lemma which is implicitly contained in \cite[Lemma~3.1]{CQW03} will be useful.

%%%%%%%%%%%%%%%%%%%%%%%%%%%%%%%%%%%%%%%%
\begin{lemma}\label{lemma.sequence}
Let $h$ be a nonnegative function in $C^1([0,\infty))$ such that there is a sequence $(r_k)_{k\ge 1}$, $r_k\to\infty$ as $k\to\infty$, for which $h(r_k)\longrightarrow 0$ as $k\to\infty$. Then, there is a sequence $(\rho_k)_{k\ge 1}$, $\rho_k\to\infty$ as $k\to\infty$, such that $h(\rho_k)\longrightarrow 0$ and $\rho_k h'(\rho_k)\longrightarrow 0$ as $k\to\infty$.
\end{lemma}
%%%%%%%%%%%%%%%%%%%%%%%%%%%%%%%%%%%%%%%%

\begin{proof}
If $h$ oscillates infinitely many times, then one can select $(\rho_k)_{k\ge 1}$ to be a sequence of local minimum points of $h$. Otherwise, if $h$ does not oscillate infinitely many times, it eventually decreases monotonically to zero and $h'$ belongs to $L^1(0,\infty)$. Then, one can select $\rho_k>k$ such that $\rho_k\ |h'(\rho_k)|\le 1/k$.  
\end{proof}

The next result provides a complete description of $A$ in terms of $R(a)$.

%%%%%%%%%%%%%%%%%%%%%%%%%%%%%%%%%%%%%%%%
\begin{lemma}\label{lemma.caractA2}
Let $a>0$. Then $a\in A$ if and only if $R(a)<\infty$.
\end{lemma}
%%%%%%%%%%%%%%%%%%%%%%%%%%%%%%%%%%%%%%%%

\begin{proof}
The proof is divided into two technical steps. Let $a\in A$ and suppose for contradiction that $R(a)=\infty$.

\noindent \textbf{Step 1. Interpretation in terms of decay.} By Lemma~\ref{lemma.caractA1}, we deduce that $w'<0$ in $(R_1(a),\infty)$ and the non-negativity of $w$ guarantees that $w$ has a limit $\ell\ge 0$ as $r\to\infty$. In addition, $w'$ belongs to $L^1(R_1(a),\infty)$ from which we deduce that there exists a sequence $r_k\to\infty$ such that $r_k w'(r_k) \longrightarrow 0$ as $k\to\infty$. Since $w'<0$, it follows from Lemma~\ref{lemma.sequence} that we may also assume that $r_k(r_k w''(r_k) + w'(r_k)) \longrightarrow 0$ as $k\to\infty$, whence $r_k^2 w''(r_k) \longrightarrow 0$ as $k\to\infty$. We evaluate \eqref{ODE2} at $r=r_k$ and pass to the limit as $k\to\infty$ to deduce that $\ell\in \{0,w^*\}$ (recall that $\eta<0$). Owing to \eqref{sup.crit}, the possibility $\ell=w^*$ is excluded and we conclude that 
\begin{equation}
\lim\limits_{r\to\infty} w(r)=0.
\label{phl15}
\end{equation}

\noindent \textbf{Step 2. The contradiction.} Since $p<2$, we infer from \eqref{phl12} and \eqref{phl15} that 
\begin{equation*}
\frac{r|f'(r)|}{f(r)}\leq C\ r f(r)^{(2-p)/p} = C\ w(r)^{(2-p)/p}\longrightarrow 0.
\end{equation*}
Consequently, there exists $r_{*}>R_1(a)$ such that
\begin{equation*}
-\mu\le \frac{rf'(r)}{f(r)} \le 0 \ \hbox{for any} \ r>r_{*},
\end{equation*}
from which we deduce that
$$
w'(r) = r^\mu\ \left( f'(r) + \mu f(r) \right) \ge 0\,, \quad r\in (r_*,\infty)\,,
$$
and a contradiction with \eqref{phl15}. Therefore $R(a)<\infty$.

Conversely, if $R(a)<\infty$, then $w(R(a))=0=w(0)$, which implies that
$w$ has a maximum point in $(0,R(a))$, hence $a\in A$.
\end{proof}

%%%%%%%%%%%%%%%%%%%%%%%%%%%%%%%%%%%%%%%%
\begin{proposition}\label{prop.A}
The set $A$ is an open interval of the form $(0,a_*)$ for some $a_*>0$.
\end{proposition}
%%%%%%%%%%%%%%%%%%%%%%%%%%%%%%%%%%%%%%%%

\begin{proof}
We first show that $A$ is non-empty and that $(0,a)\subset A$ for $a$ sufficiently
small. The proof is the same as that of \cite[Theorem~2]{Pe042} and we only sketch it here for
the sake of completeness. We perform a rescaling in the $r$-variable and, for $a>0$, we define $g(.;a)$ by 
$$
f(r;a) = a g\left( r a^{(2-p)/p} ; a \right)\,, \quad r\in [0,R(a))\,.
$$
Then $g=g(.;a)$ solves
\begin{equation}\label{eqresc.1}
\left\{\begin{array}{ll}(|g'|^{p-2}g')'(s)+\displaystyle{\frac{N-1}{s}} (|g'|^{p-2}g')(s) +\b
s g'(s)+\a g(s)-a^{(2q-p)/p} |g'(s)|^q=0, \\ 
 \\
g(0)=1, \ g'(0)=0.\end{array}\right.
\end{equation}
Owing to Lemma~\ref{lemma.basic}, we have 
\begin{equation}
g(s)>0 \;\;\mbox{ and }\;\; - \a^{1/q}\ a^{-(2q-p)/pq} \le g'(s) < 0 \;\;\mbox{ for }\;\; s\in \left( 0 , R(a) a^{(2-p)/p} \right)\,.
\label{phl16}
\end{equation}

We next study the limit of \eqref{eqresc.1} as $a\to 0$ which reads, since $q>p/2$, 
\begin{equation}
\left\{\begin{array}{ll}(|h'|^{p-2}h')'(s)+\displaystyle{\frac{N-1}{s}} (|h'|^{p-2}h')(s) +\b
sh'(s)+\a h(s)=0, \\ 
 \\
h(0)=1, \ h'(0)=0.\end{array}\right.
\label{phl17}
\end{equation}
Arguing as in Lemma~\ref{lemma.basic}, there is $S_0>0$ such that the solution $h$ to \eqref{phl17} satisfies $h(S_0)=0$, $h(s)>0$ and $h'(s)<0$ for $s\in (0,S_0)$. Moreover, proceeding as in the proof of \cite[Theorem~2]{Pe042}, we show that $S_0<\infty$ with $h'(S_0)<0$. It is then easy to show, using the continuous dependence of solutions for \eqref{eqresc.1}, that for $a>0$
sufficiently small, $g(.;a)$ also vanishes and thus $R(a)<\infty$. Consequently, thanks to Lemma~\ref{lemma.caractA2}, $(0,a)\subset A$ for $a$ small enough.

It remains to show that $A$ is an open interval. By Lemma~\ref{lemma.caractA1}, if $a\in
A$, we have $w'(R_1(a);a)=0$ with $w''(R_1(a);a)<0$. By continuous dependence, this property readily implies that $A$ is open. In addition, we can apply the implicit function theorem and conclude that $a\mapsto R_1(a)$ belongs to $C^1(A)$. Define 
$m(a):=w(R_1(a);a)$ for $a\in A$. Then $m\in C^1(A)$ and we infer from Lemma~\ref{monotonicity} (with $r_0=R_1(a)$) and Lemma~\ref{lemma.caractA1} that 
\begin{equation}
\frac{dm}{da}(a)=w'(R_1(a);a) \frac{dR_1}{da}(a)+ \partial_a w(R_1(a);a) = \partial_a w(R_1(a);a) >0\,, \quad a\in A \,.
\label{phl18}
\end{equation} 
Consider now $a_2\in A$ and define $a_1\ge 0$ by 
$$
a_1 := \sup{\{ a \in (0,a_2)\ :\ a\not\in A \}}\,.
$$
Since $A$ is open, we have $a_1<a_2$, $a_1\not\in A$,  and $(a_1,a_2)\subset A$. Setting
$$
\rho := \liminf_{a\searrow a_1} R_1(a) \in [0,\infty]\,,
$$
there are three possibilities:

If $\rho=\infty$, we actually have $R_1(a)\longrightarrow \infty$ as $a\searrow a_1$ and it follows by continuous dependence that, for any $r>0$, we have $r<R_1(a)$ for $a>a_1$ close enough to $a_1$ and thus
$$
w'(r;a_1) = \lim_{a\searrow a_1} w'(r;a) \ge 0\,.
$$
In addition, by Lemma~\ref{lemma.caractA2} and \eqref{phl18},
\begin{equation}\label{crit.sup2}
w(r;a_1) =\lim_{a\searrow a_1} w(r;a) \le m((a_1+a_2)/2)<w^*,
\end{equation}
so that $w(.;a_1)$ is a non-decreasing and bounded function. It thus has a limit $\ell\in [0, m((a_1+a_2)/2)]$ as $r\to\infty$. If $\ell>0$, we argue as in Step~1 of the proof of Lemma~\ref{lemma.caractA2} to obtain that $\ell=w^*$ and a contradiction. Consequently, $\ell=0$ so that $w(.;a_1)\equiv 0$ and thus $a_1=0$.

If $\rho\in (0,\infty)$, there is a sequence $(a^j)\in (a_1,a_2)$ such that $a^j\longrightarrow a_1$ and $R_1(a^j)\longrightarrow \rho$ as $j\to\infty$. By continuous dependence we have
$$
w'(\rho;a_1) = \lim_{j\to\infty} w'(R_1(a^j);a^j) = 0\,,
$$
whence $a_1\in A$ and a contradiction.

If $\rho=0$, there is a sequence $(a^j)\in (a_1,a_2)$ such that $a^j\longrightarrow a_1$ and $R_1(a^j)\longrightarrow 0$ as $j\to\infty$. Let us assume for contradiction that $a_1>0$. Then there is $r_1>0$ such that $w'(r;a_1)>0$ for $r\in (0,r_1)$ by Lemma~\ref{lemma.expansion}.  Given $r\in (0,r_1)$, we have $R_1(a^j)<r_1$ for $j$ large enough whence 
$$
w'(r;a_1) = \lim_{j\to\infty} w'(r;a^j) \le 0\,,
$$
by continuous dependence, and a contradiction. Consequently, $a_1=0$ in this case as well.

We have thus shown that, given $a_2\in A$, the interval $(0,a_2)$ is included in $A$ from which Proposition~\ref{prop.A} follows.
\end{proof}

%%%%%%%%%%%%%%%%%%%%%%%%%%%%%%%%%%%%%%%%
%%%%%%%%%%%%%%%%%%%%%%%%%%%%%%%%%%%%%%%%
\subsection{Characterization of $C$}\label{sec:cc}
%%%%%%%%%%%%%%%%%%%%%%%%%%%%%%%%%%%%%%%%
%%%%%%%%%%%%%%%%%%%%%%%%%%%%%%%%%%%%%%%%

We begin with the following useful result.

%%%%%%%%%%%%%%%%%%%%%%%%%%%%%%%%%%%%%%%%
\begin{lemma}\label{lemma.caractC1}
Let $a>0$. Then $a\in C$ if and only if
\begin{equation}
\sup\limits_{r\in(0,R(a))}w(r;a)>w^*\,.
\label{phl19}
\end{equation}
\end{lemma}
%%%%%%%%%%%%%%%%%%%%%%%%%%%%%%%%%%%%%%%%

\begin{proof}
The direct implication is obvious. Conversely, if $a>0$ is such that \eqref{phl19} is satisfied, it follows from Lemma~\ref{lemma.caractA1} that $a\not\in A$. Then $a\in B\cup C$ and $w(.;a)$ is an increasing function in $(0,\infty)$. If it is bounded, it has a finite limit $\ell$ as $r\to\infty$ and we argue as in the first step of the proof of Lemma~\ref{lemma.caractA2} to deduce that $\ell\in \{0, w^*\}$, clearly contradicting \eqref{phl19}. Consequently, $w(.;a)$ is unbounded and $a\in C$.
\end{proof}

%%%%%%%%%%%%%%%%%%%%%%%%%%%%%%%%%%%%%%%%
\begin{proposition}\label{prop.C}
The set $C$ is an open interval of the form $(a^*,\infty)$.
\end{proposition}
%%%%%%%%%%%%%%%%%%%%%%%%%%%%%%%%%%%%%%%%

\begin{proof}
Let us first show that $C$ is non-empty. Given $a>0$, by Lemma~\ref{lemma.basic} we
have $f'(r)\geq-(\a a)^{1/q}$ for $r\in (0,R(a))$, whence
\begin{equation}
f(r) \ge a - (\alpha a)^{1/q} r \quad \hbox{for} \ r\in (0,R(a)).
\label{phl20}
\end{equation}
In particular, $R(a)\geq a(\a a)^{-1/q}=\a^{-1/q}a^{(q-1)/q}$. Consequently, $\alpha^{-1/q} a^{(q-1)/q}/2
\in (0,R(a))$ and it follows from \eqref{phl20} that 
$$
w\left( \alpha^{-1/q} a^{(q-1)/q}/2 ; a \right) \geq \frac{a}{2}(\alpha^{-1/q}
a^{(q-1)/q}/2)^{\mu} = \a^{-p/q(2-p)}\ 2^{-2/(2-p)}\ a^{(2q-p)/q(2-p)},
$$
which can be chosen greater that $2w^*$ for $a$ large enough since
$q>p/2$, and guarantees that $a\in C$ for $a$ large enough,
according to Lemma~\ref{lemma.caractC1}. This shows that $C$ is
non-empty and that $a\in C$ for $a$ large enough. In addition, $C$ is clearly an open set thanks to Lemma~\ref{lemma.caractC1}.

Now, let $a_0\in C$. Then $w'(.;a_0)>0$ in $(0,\infty)$ and also $\partial_a w(.;a_0)>0$ in $(0,\infty)$ by Lemma~\ref{monotonicity}. Therefore, $C$ is an open interval.
\end{proof}

We now show that the set $C$ is composed in fact of the \emph{slow orbits}.

%%%%%%%%%%%%%%%%%%%%%%%%%%%%%%%%%%%%%%%%
\begin{proposition}\label{prop.caractC}
For every $a\in C$, there exists $k(a)\in(0,\infty)$ such that
\begin{equation*}
\lim\limits_{r\to\infty}r^{(p-q)/(q-p+1)}f(r;a)=k(a).
\end{equation*}
\end{proposition}
%%%%%%%%%%%%%%%%%%%%%%%%%%%%%%%%%%%%%%%%

While the outcome of Proposition~\ref{prop.caractC} is similar to that of \cite[Theorem 4.1]{CQW03}, only the first step of the proof of Proposition~\ref{prop.caractC} borrows some ideas of the proof of \cite[Theorem 4.1]{CQW03}, that is, the introduction of a new variable and a new unknown function. The other two steps are different owing to the different nature of the absorption on the one hand (Step~2) and the fact that $q$ takes values below $1$ on the other hand (Step~3). Concerning the latter, we prove a general lemma.

%%%%%%%%%%%%%%%%%%%%%%%%%%%%%%%%%%%%%%%%
\begin{lemma}\label{lemma.ODEs}
Let $X\in C^1([0,\infty))$ satisfy the differential inequality
\begin{equation}\label{difineq}
X'(t)+C_5 e^{\gamma t}X(t)\leq C_6(1+e^{\delta t}), \quad t>0\,, 
\end{equation}
for some $0<\delta<\gamma$ and $C_5>0$, $C_6>0$. Given $\theta\in (0,\gamma-\delta)$, there exist $t_\theta>0$ and $K_\theta>0$ depending only on $\gamma$, $\delta$, $C_5$, $C_6$, $X(0)$, and $\theta$ such that $X(t)\leq K_\theta e^{-\theta t}$ for all $t\ge t_\theta$.
\end{lemma}
%%%%%%%%%%%%%%%%%%%%%%%%%%%%%%%%%%%%%%%%

\begin{proof}
Make the change of variable $X(t)=Y(e^{\gamma t})=Y(\tau)$ for $t\ge 0$, where
$\tau:=e^{\gamma t}\ge 1$. Then the inequality \eqref{difineq} becomes
$$
\gamma Y'(\tau)+C_5Y(\tau)\leq \frac{C_6}{\tau}\ \left( 1 +\tau^{\delta/\gamma}\right), \quad \tau\ge 1,
$$
or, equivalently, after a straightforward transformation
$$
\frac{d}{d\tau}\left( Y(\tau)e^{C_7\tau} \right)\leq
C_6\left(\frac{1}{\tau}+\frac{1}{\tau^{(\gamma-\delta)/\gamma}}\right)e^{C_7\tau}, \quad \tau\ge 1,
$$
with $C_7:=C_5/\gamma$. After integration over $(1,\tau)$, $\tau\ge 1$, we obtain
\begin{equation}\label{eqint}
Y(\tau)\leq
Y(1)e^{-C_7(\tau-1)}+C_6\int_1^{\tau}\left(\frac{1}{\sigma}+\frac{1}{\sigma^{(\gamma-\delta)/\gamma}}\right)e^{C_7(\sigma-\tau)}\,d\sigma, \quad \tau\ge 1.
\end{equation}
We take $\tau\ge 2$ sufficiently large such that $\tau/\log{\tau}\geq 2/C_7$. In order to estimate the integral in the
right-hand side of \eqref{eqint}, we split it and handle differently the contributions to the integral of the intervals $(1,\tau/2)$, $(\tau/2, \tau-(\log{\tau})/C_7)$ and  $(\tau-(\log{\tau})/C_7,\tau)$. First,
$$
\int_1^{\tau/2}\left(\frac{1}{\sigma}+\frac{1}{\sigma^{(\gamma-\delta)/\gamma}}\right) e^{C_7(\sigma-\tau)}\,d\sigma \le 2 \int_1^{\tau/2} e^{-C_7 \tau/2}\, d\sigma \le \tau\ e^{-C_7 \tau/2}.
$$
Next, $\sigma-\tau \le -(\log{\tau})/C_7$ for $\sigma\in (\tau/2, \tau-(\log{\tau})/C_7)$ so that
\begin{eqnarray*}
\int_{\tau/2}^{\tau-(\log{\tau})/C_7} \left(\frac{1}{\sigma}+\frac{1}{\sigma^{(\gamma-\delta)/\gamma}}\right) e^{C_7(\sigma-\tau)}\,d\sigma & \le & \frac{1}{\tau}\ \int_{\tau/2}^{\tau-(\log{\tau})/C_7} \left(\frac{1}{\sigma}+\frac{1}{\sigma^{(\gamma-\delta)/\gamma}}\right)\,d\sigma \\
& \le & \frac{1}{\tau}\ \left[ \log{\sigma} + \frac{\gamma}{\delta}\ \sigma^{\gamma/\delta} \right]_{\sigma=\tau/2}^{\sigma=\tau} \\
& \le & \frac{\gamma}{\delta}\ \left( \frac{1}{\tau}+\frac{1}{\tau^{(\gamma-\delta)/\gamma}} \right).
\end{eqnarray*}
Finally, if $\sigma\in (\tau-(\log{\tau})/C_7,\tau)$, we have $\sigma\ge \tau/2$ and $\sigma-\tau<0$ and 
\begin{eqnarray*}
\int_{\tau-(\log{\tau})/C_7}^{\tau} \left(\frac{1}{\sigma}+\frac{1}{\sigma^{(\gamma-\delta)/\gamma}}\right) e^{C_7(\sigma-\tau)}\,d\sigma & \le & \int_{\tau-(\log{\tau})/C_7}^{\tau} \left(\frac{2}{\tau} + \left( \frac{2}{\tau} \right)^{(\gamma-\delta)/\gamma} \right) \,d\sigma\\
& \le & \frac{2 \log{\tau}}{C_7}\ \left( \frac{1}{\tau} + \frac{1}{\tau^{(\gamma-\delta)/\gamma}} \right).
\end{eqnarray*}
Plugging these estimates into \eqref{eqint}, we obtain that, given $\varrho\in (0,(\gamma-\delta)/\delta)$, there are $K_\varrho>0$ and $\tau_\varrho\ge 2$ such that $Y(\tau)\leq K_\varrho \tau^{-\varrho}$ for all $\tau\ge \tau_\varrho$. Coming back to $X$ gives the claim.
\end{proof}

The proof of Proposition~\ref{prop.caractC} is divided into several steps.

\begin{proof}[Proof of Proposition~\ref{prop.caractC}]
\noindent \textbf{Step 1. Introducing a new function.} Fix $a\in C$. As in \cite[Theorem~4.1]{CQW03}, we define a new function $\Lambda$ by
\begin{equation}
f(e^{\tau};a)=f(1;a)\exp\left(-\int_{0}^{\tau}\Lambda(s)\,\ds\right)\,, \quad \tau\in\real.
\label{spirou}
\end{equation}
Then
\begin{equation*}
f'(e^{\tau};a)=-e^{-\tau}\Lambda(\tau)f(e^{\tau};a), \quad
f''(e^{\tau};a)=e^{-2\tau}f(e^{\tau};a)[\Lambda(\tau)+\Lambda(\tau)^2-\Lambda'(\tau)],
\end{equation*}
and it follows from \eqref{ODE2} that $\Lambda$ solves the following differential equation
\begin{equation}\label{ODE4}
(p-1)\Lambda'(\tau)=F(\tau,\Lambda(\tau))\,, \quad \tau\in\real\,,
\end{equation}
with
$$
F(\tau,\Lambda):=(p-1)\Lambda^2+(p-N)\Lambda+\Lambda^{2-p} w(e^{\tau};a)^{2-p}\ \left[ \a-\b\Lambda-\Lambda^qe^{-q\tau}f(e^{\tau};a)^{q-1} \right]\,.
$$
Since $a\in C$, we have $R(a)=\infty$ and it follows from Lemma~\ref{lemma.basic} that $f'(.;a)<0$, whence $\Lambda>0$. In addition, for $\tau\in\real$,
$$
w'(e^{\tau};a)=e^{(\mu-1)\tau} \left( e^{\tau} f'(e^{\tau};a) + \mu f(e^{\tau};a) \right) = e^{(\mu-1)\tau} (\mu-\Lambda) f(e^{\tau};a) > 0
$$
since $a\in C$, from which we deduce that $\Lambda(\tau)<\mu$ for all $\tau\in\real$.

\noindent \textbf{Step 2. Limit of $\Lambda$ as $\tau\to\infty$.} We claim that
\begin{equation}\label{limlam}
\lim\limits_{\tau\to\infty}e^{-q\tau}f(e^{\tau};a)^{q-1}=0.
\end{equation}
Indeed, either $q\geq1$ and \eqref{limlam} is obvious thanks to the
boundedness of $f$ which guarantees that
$$
e^{-q\tau}f(e^{\tau};a)^{q-1} \le a^{q-1}\ e^{-q\tau}\,, \quad \tau\in\real\,,
$$ 
or $q\in(p/2,1)$ and
$$
e^{-q\tau}f(e^{\tau};a)^{q-1}=e^{-q\tau}e^{-\mu(q-1)\tau}w(e^{\tau};a)^{q-1}=e^{-(2q-p)\tau/(2-p)}w(e^{\tau};a)^{q-1},
$$
and both terms converge to zero as $\tau\to\infty$ since $a\in C$ and
$q\in(p/2,1)$. Hence \eqref{limlam} holds true. Consequently, since $a\in C$, we have
$$
\lim\limits_{\tau\to\infty}F(\tau,\gamma)=+\infty \ \hbox{if} \
\gamma<\frac{\a}{\b}, \quad
\lim\limits_{\tau\to\infty}F(\tau,\gamma)=-\infty \ \hbox{if} \
\gamma>\frac{\a}{\b},
$$
which ensures that
\begin{equation}\label{limlam2}
\lim\limits_{\tau\to\infty}\Lambda(\tau)=\frac{\a}{\b}.
\end{equation}

\noindent \textbf{Step 3. Exponential decay of $\Lambda-\a/\b$.} We study the rate of
convergence of $\Lambda(\tau)$ to $\a/\b$ as $\tau\to\infty$. This
last step of the proof is done in \cite{CQW03} by construction of
suitable sub- and supersolutions. We give here an alternative and
direct approach based on Lemma~\ref{lemma.ODEs} which allows us to handle also the case $q<1$.

Since $\Lambda(\tau)\to\a/\b$ as $\tau\to\infty$, taking into
account the definition of $\Lambda$, we find
$$
\lim\limits_{r\to\infty}\frac{rf'(r)}{f(r)}=-\frac{\a}{\b},
$$
hence, for any $\e\in (0,\a/\b)$, there exists $r_{\e}\ge 1$ such that
\begin{equation}
-\frac{\a}{\b}-\e\leq\frac{rf'(r)}{f(r)} = - \Lambda(\log{r}) \leq-\frac{\a}{\b}+\e, \quad r>r_{\e}.
\label{phl21}
\end{equation}
By integrating, we obtain that there exists a constant $K_{\e}>0$ such that
\begin{equation}\label{corBPT}
K_{\e}r^{-(\a/\b)-\e}\leq f(r)\leq K_{\e}r^{-(\a/\b)+\e}, \quad r>r_{\e}.
\end{equation}
We then come back to the differential equation \eqref{ODE4} and check that the conditions required to
apply Lemma~\ref{lemma.ODEs} with either $X=(\Lambda-\a/\b)_+$ or $X=(\a/\b-\Lambda)_+$ are fulfilled. Indeed, we can
write \eqref{ODE4} in the form
\begin{equation}\label{ODE5}
(p-1) \Lambda'(\tau) + \b\Lambda(\tau)^{2-p} w(e^\tau;a)^{2-p}\ \left(\Lambda(\tau) - \frac{\a}{\b} \right) = S(\tau)\,, \quad \tau\in\real\,,
\end{equation}
with
$$
S(\tau) := (p-1)\Lambda(\tau)^2+(p-N)\Lambda(\tau) - \Lambda(\tau)^{q-p+2} w(e^\tau;a)^{2-p} e^{-q\tau} f(e^\tau;a)^{q-1}\,.
$$ 
On the one hand, owing to the boundedness of $\Lambda$, the non-negativity of $f(.;a)$, $w(.;a)$, and $\Lambda$, and \eqref{corBPT}, we have
\begin{equation}
S(\tau) \le (p-1)\mu^2 + p \mu\,, \quad \tau\in\real\,,
\label{phl22}
\end{equation}
and 
\begin{eqnarray}
S(\tau) & \ge & - N \mu - \mu^{q-p+2}\ e^{(p-q)\tau}\ f(e^\tau;a)^{q-p+1} \nonumber \\
& \ge & - N \mu - \mu^{q-p+2}\ e^{(p-q)\tau}\ K_\e^{q-p+1}\ e^{(\e(q-p+1)+q-p)\tau} \nonumber\\
S(\tau)  & \ge & - N \mu - \mu^{q-p+2}\ K_\e^{q-p+1}\ e^{\e(q-p+1)\tau}\,, \quad \tau\ge \log{r_\e}\,. \label{phl23}
\end{eqnarray}
On the other hand, it follows from \eqref{limlam2}, \eqref{phl21}, and \eqref{corBPT} that 
\begin{equation}
\beta \Lambda(\tau)^{2-p}\ w(e^{\tau};a)^{2-p} \ge \beta\ \left( \frac{\a}{\b} - \e \right)^{2-p}\ K_\e^{2-p}\ e^{(2-p)(\mu-(\a/\b)-\e)\tau}\,, \quad \tau\ge \log{r_\e}\,.
\label{phl24}
\end{equation}

We now fix $\e\in (0,\a/\b)$ such that 
$$
\e \le \frac{\a}{2\b} \;\;\mbox{ and }\;\; 2\e (q-p+1) < \frac{1}{\b} - \e (2-p) = (2-p) \left( \mu-\frac{\a}{\b}-\e \right)\,,
$$ 
such a choice being always possible since $\mu>\a/\b$. Introducing $X_1:=(\Lambda-\a/\b)_+$, we infer from \eqref{ODE5} and \eqref{phl22} that 
$$
(p-1) X_1'(\tau) + \b\Lambda(\tau)^{2-p} w(e^\tau;a)^{2-p}\ X_1(\tau) \le S(\tau)_+ \le (p-1)\mu^2 + p \mu\,,
$$
while \eqref{phl24}, the non-negativity of $X_1$, and the choice of $\e$ guarantee that
$$
(p-1) X_1'(\tau) + \b\ \left( \frac{\a K_\e}{2\b} \right)^{2-p}\ e^{(2-p)(\mu-(\a/\b)-\e)\tau}\ X_1(\tau) \le S(\tau)_+ \le (p-1)\mu^2 + p \mu
$$
for $\tau\ge \log{r_\e}$. Consequently, $X_1$ satisfies a differential inequality of the form of
\eqref{difineq} with $\gamma=(\mu-(\a/\b)-\e)(2-p)>0$ and $\delta=0$. We then apply Lemma~\ref{lemma.ODEs} and conclude that  
$$
\left (\Lambda(\tau)-\frac{\a}{\b} \right)_+ \leq C\ e^{-\e (q-p+1) \tau}
$$
for $\tau$ large enough. Similarly, setting $X_2:= (\a/\b - \Lambda)_+$, it follows from \eqref{ODE5} and \eqref{phl23} that 
\begin{eqnarray*}
(p-1) X_2'(\tau) + \b\Lambda(\tau)^{2-p} w(e^\tau;a)^{2-p}\ X_2(\tau) & \le & \left( - S(\tau) \right)_+ \\
& \le & N \mu + \mu^{q-p+2}\ K_\e^{q-p+1}\ e^{\e(q-p+1)\tau}\,,
\end{eqnarray*}
and from \eqref{phl24}, the non-negativity of $X_1$, and the choice of $\e$ that
\begin{equation*}
\begin{split}
(p-1) X_2'(\tau) & + \b\ \left( \frac{\a K_\e}{2\b} \right)^{2-p}\ e^{(2-p)(\mu-(\a/\b)-\e)\tau}\ X_2(\tau) \\
& \le N \mu + \mu^{q-p+2}\ K_\e^{q-p+1}\ e^{\e(q-p+1)\tau}
\end{split}
\end{equation*}
for $\tau\ge \log{r_\e}$. Consequently, $X_2$ satisfies a differential inequality of the form of \eqref{difineq} with $\gamma=(\mu-(\a/\b)-\e)(2-p)$ and $\delta=\e(q-p+1)$, and the choice of $\e$ guarantees that $\gamma-\delta\ge \delta>0$. We then apply Lemma~\ref{lemma.ODEs} and conclude that  
$$
\left (\frac{\a}{\b} - \Lambda(\tau) \right)_+ \le C\ e^{-\e (q-p+1) \tau}
$$
for $\tau$ large enough. We have thus established that $|\Lambda(\tau)-\a/\b|\le C\ e^{-\e (q-p+1) \tau}$ for $\tau$ large enough. In particular, $\Lambda-\a/\b$ belongs to $L^1(1,\infty)$ and, recalling \eqref{spirou}, we realize that 
\begin{equation*}
\begin{split}
r^{\a/\b}f(r;a)&=f(1;a)\exp\left(-\int_0^{\log\,r}\left(\Lambda(\tau)-\frac{\a}{\b}\right)\,d\tau\right)\\
&\mathop{\longrightarrow }_{r\to\infty} f(1;a)\exp\left(-\int_0^{\infty}\left(\Lambda(\tau)-\frac{\a}{\b}\right)\,d\tau\right)=:k(a)\in(0,\infty)\,,
\end{split}
\end{equation*}
as stated in Proposition~\ref{prop.caractC}.
\end{proof}

%\noindent \textbf{Remark.} One can prove that
%$\lim\limits_{a\to\infty}k(a)=\infty$ and $\lim\limits_{a\to
%a^*}k(a)=0$. This can be done as in \cite{CQW03} but we omit the
%details, since it will not be used in the sequel.

%%%%%%%%%%%%%%%%%%%%%%%%%%%%%%%%%%%%%%%%
%%%%%%%%%%%%%%%%%%%%%%%%%%%%%%%%%%%%%%%%
\subsection{Characterization of $B$. Uniqueness}\label{sec:cbu}
%%%%%%%%%%%%%%%%%%%%%%%%%%%%%%%%%%%%%%%%
%%%%%%%%%%%%%%%%%%%%%%%%%%%%%%%%%%%%%%%%

Since $A=(0,a_*)$ and $C=(a^*,\infty)$, we already know that $B$ is
non-empty, that is, there exists at least one radially symmetric self-similar very
singular solution. We complete the proof of Theorem~\ref{th:main} by the
uniqueness result.

%%%%%%%%%%%%%%%%%%%%%%%%%%%%%%%%%%%%%%%%
\begin{proposition}\label{prop.uniq}
We have $a_*=a^*$, thus $B$ contains only one element.
\end{proposition}
%%%%%%%%%%%%%%%%%%%%%%%%%%%%%%%%%%%%%%%%

\begin{proof}
Let $a\in B$. Then, $R(a)=\infty$, $w(.;a)$ is increasing and has a finite limit as $r\to\infty$. Arguing as in the first step of the proof of Lemma~\ref{lemma.caractA2}, we deduce that 
\begin{equation}
\lim\limits_{r\to\infty} r^{\mu}f(r;a) = \lim\limits_{r\to\infty} w(r;a)=w^*.
\label{phl30}
\end{equation}
As in \cite{Pe042}, the proof is divided into two steps, with a different argument to prove the first step.

\noindent \textbf{Step 1.} We show that $rw'(r;a)\longrightarrow 0$ as
$r\to\infty$ for $a\in B$. Indeed, we notice that the
differential equation \eqref{IVPa} also reads
\begin{equation*}
\frac{d}{dr}(r^{N-1}|f'(r)|^{p-2}f'(r)+\b
r^Nf(r))=r^{N-1}(|f'(r)|^q-(\a-\b N)f(r))\,, \quad r\ge 0\,.
\end{equation*}
Since $q>p/2$, it follows from Lemma~\ref{lemma.basic} and Proposition~\ref{prop.gradient} that, as $r\to\infty$, 
$$
|f'(r;a)|^q\leq C_4\ f(r;a)^{2q/p}=o(f(r;a)),
$$
and we deduce from \eqref{phl30} that, as $r\to\infty$,
\begin{equation*}
\frac{d}{dr}(r^{N-1}|f'(r)|^{p-2}f'(r)+\b r^Nf(r))\sim-(\a-\b
N)w^*r^{N-1-\mu}.
\end{equation*}
Since $\mu>N$, we obtain after integration
\begin{equation*}
-r^{N-1}(-f'(r;a))^{p-1}+\b r^Nf(r;a)\sim(\a-\b
N)w^*\frac{r^{-(\mu-N)}}{\mu-N},
\end{equation*}
hence, using again \eqref{phl30},
$$
-(-f'(r;a))^{p-1}\sim-\b rf(r;a)+\frac{\a-\b
N}{\mu-N}w^*r^{-(\mu-1)} \sim \left(\frac{\a-\b
N}{\mu-N}-\b \right) w^* r^{-(\mu-1)}
$$
as $r\to\infty$. Noticing that
$$
\frac{\a-\b N}{\mu-N}-\b=-\frac{1}{p(N+1)-2N}<0,
$$
we conclude that, as $r\to\infty$,
$$
f'(r;a)\sim-\left(\frac{w^*}{p(N+1)-2N}\right)^{1/(p-1)}r^{-2/(2-p)}.
$$
Consequently,
\begin{equation*}
\begin{split}
\lim_{r\to\infty} rw'(r;a)&= \lim_{r\to\infty} \left( r^{2/(2-p)} f'(r;a) + \mu r^{p/(2-p)} f(r;a) \right)\\
&= -\left(\frac{w^*}{p(N+1)-2N}\right)^{1/(p-1)}+\mu w^*,
\end{split}
\end{equation*}
But from the definition and a simple computation we find
$$
w^*=\mu^{(p-1)/(2-p)}(p(N+1)-2N)^{1/(2-p)},
$$
that is
$$
\mu=(w^*)^{(2-p)/(p-1)}(p(N+1)-2N)^{-1/(p-1)},
$$
hence
\begin{equation*}
\mu w^*=(w^*)^{1+(2-p)/(p-1)}(p(N+1)-2N)^{-1/(p-1)}=\left(\frac{w^*}{p(N+1)-2N}\right)^{1/(p-1)}.
\end{equation*}
We have thus shown that
\begin{equation}
\lim_{r\to\infty} rw'(r;a) = 0\,.
\label{phl31}
\end{equation}

\medskip

\noindent \textbf{Step 2.} We now follow the argument in \cite[Theorem~5.1]{CQW03}. Let $a\in B$. Since $w(r;a)\longrightarrow w^*$ and $r w'(r;a) \longrightarrow 0$ as $r\to\infty$ by \eqref{phl30} and \eqref{phl31}, the homogeneous linear operator $L_a$ defined by \eqref{linearized} can be written as
\begin{equation*}
L_a(\varphi)=r^2 \varphi''(r) + P_1(r)\ r\ \varphi'(r) + P_0(r)\ \varphi(r)\,,
\end{equation*}
with 
\begin{equation}
\begin{split}
\lim\limits_{r\to\infty} P_1(r) & = b := \frac{N-1}{p-1} - 2\mu + \frac{p(N+1)-2N}{p-1}\ (1+\beta\mu)\,, \\ 
\lim\limits_{r\to\infty} P_0(r) & = -c\,, \quad c:= \mu \frac{p(N+1)-2N}{p-1} > 0\,.
\end{split}
\label{asterix}
\end{equation}
Owing to the negativity of $c$, we claim that, given two linearly independent solutions of the equation $L_a(\varphi)=0$, one of them has to be unbounded as $r\to\infty$. Taking this result for granted, we define the function $\psi$ as the solution to $L_a(\psi)=0$ in $(1,\infty)$ with $\psi(1)=0$, $\psi'(1)=1$. Arguing as in Lemma~\ref{monotonicity} with $r_0=\infty$ (recall that $R(a)=\infty$ as $a\in B$), we conclude that there is a positive real number $\kappa>0$ such that $\partial_a w(.;a)>\kappa \psi>0$ in $(1,\infty)$. Since $L_a(\partial_a w(.;a))=0$ with $\partial_a w(1;a)>0$ by Lemma~\ref{lemma.monot2} and Lemma~\ref{monotonicity}, $\partial_a w(.;a)$ and $\psi$ are clearly linearly independent solutions of the equation $L_a(\varphi)=0$. The above claim implies that one of them must be unbounded and thus that $\partial_a w(.;a)$ is unbounded since it dominates $\psi$.

We then deduce that, for all $a\in B$, $\partial_a w(r;a)\longrightarrow\infty$ as $r\to\infty$. Therefore, if
$a_*<a^*$, we can apply Fatou's lemma to get
\begin{equation*}
0=\lim\limits_{r\to\infty}[w(r,a^*)-w(r,a_*)]=\lim\limits_{r\to\infty}\int\limits_{a_*}^{a^*} \partial_a w(r;a)\, da \geq \int\limits_{a_*}^{a^*} \liminf\limits_{r\to\infty} \partial_a w(r;a) \,da = \infty,
\end{equation*}
which is a contradiction. Hence $a_*=a^*$.

Let us finally sketch the proof of the above claim. Introducing the new variable $\tau=\log{r}$, the linear operator $L_a$ transforms into the linear operator
$$
\tilde{L}_a(\psi)(\tau) := \tilde{\varphi}''(\tau) + \tilde{P}_1(\tau)\ \tilde{\varphi}'(\tau) + \tilde{P}_0(\tau)\ \tilde{\varphi}(\tau)
$$
with $\tilde{P}_1(\tau):= P_1(e^\tau)-1$ and $\tilde{P}_0(\tau):=P_0(e^\tau)$. Introducing next the change of function 
$$
\psi(\tau) = \exp{\left( \frac{1}{2}\ \int_0^\tau \tilde{P}_1(s)\, ds \right)}\ \tilde{\varphi}(\tau)\,, \quad \tau\ge 0\,,
$$
we obtain the canonical form $K_a$ of the operator $\tilde{L}_a$ which reads
$$
K_a(\psi)(\tau) := \psi''(\tau) + Q(\tau)\ \psi(\tau) \;\;\mbox{ with }\;\; Q := \tilde{P}_0 - \frac{\tilde{P_1}^2}{4} - \frac{\tilde{P_1}'}{2}\,.
$$
At this point, we note that \eqref{asterix} ensures that
\begin{equation}
\lim\limits_{\tau\to\infty} Q(\tau) = - \lambda_0^2 := -c - \frac{b^2}{4} < 0\,.
\end{equation}
Then, if $\psi$ is a solution to $K_a(\psi)=0$, the variation of constants formula implies that there are $\psi_0\in\real$ and $\psi_1\in\real$ such that 
$$
\psi(\tau) = \psi_0\ e^{-\lambda_0 \tau} + \psi_1\ e^{\lambda_0 \tau} + \int_0^\tau \left\{ \frac{Q(s)\psi(s)}{2\lambda_0}\ \left( e^{\lambda_0(\tau-s)}-e^{-\lambda_0(\tau-s)} \right) \right\}\ ds\,, \quad \tau\ge 0\,.
$$ 
Arguing by contradiction, it is now easy to show that, if $\psi_1>0$ and $\lambda\in (0,\lambda_0)$, then the function $\tau\mapsto e^{-\lambda\tau} \psi(\tau)$ cannot be bounded. Coming back to $L_a$ gives the expected result after choosing $\lambda$ appropriately close to $\lambda_0$.
\end{proof}

%%%%%%%%%%%%%%%%%%%%%%%%%%%%%%%%%%%%%%%%
%%%%%%%%%%%%%%%%%%%%%%%%%%%%%%%%%%%%%%%%
\appendix
\section{Viscosity solutions to \eqref{FDE}}\label{sec:vs}
%%%%%%%%%%%%%%%%%%%%%%%%%%%%%%%%%%%%%%%%
%%%%%%%%%%%%%%%%%%%%%%%%%%%%%%%%%%%%%%%%

We recall here the definition of viscosity solutions for singular
diffusion equations such as \eqref{FDE} as introduced in \cite{IS, OS}. Owing to the singularity of the diffusion, it differs from the usual definition of vicosity solutions (see, e.g., \cite{CIL92}) by restricting the class of comparison functions. More precisely, let $\Xi$ be the set of functions $\xi\in C^{2}([0,\infty))$ satisfying
\begin{equation}\label{restriction.visc}
\xi(0)=\xi'(0)=\xi''(0)=0, \ \xi''(r)>0 \ \hbox{for} \ \hbox{all} \ r>0,
\quad \lim\limits_{r\to 0}|\xi'(r)|^{p-2}\xi''(r)=0.
\end{equation}
For example, $\xi(r)=r^{\sigma}$ with $\sigma>p/(p-1)>2$ belongs to
$\Xi$. In fact, if $\xi\in\Xi$, it readily follows from \eqref{restriction.visc} that
\begin{equation}
\lim_{r\to 0} \frac{\xi(r)}{r^{p/(p-1)}} = 0.
\label{phl0}
\end{equation}
We next introduce the class $\ca$ of admissible comparison
functions $\psi\in C^2(Q_{\infty})$ defined as follows: $\psi\in\ca$
if, for any $(t_0,x_0)\in Q_{\infty}$ where $\nabla\psi(t_0,x_0)
=0$, there exist a constant $\delta>0$, a function $\xi\in\Xi$, and a
nonnegative function $\omega\in C([0,\infty))$ satisfying $\omega(r)/r\longrightarrow 0$ as $r\to 0$,
such that, for all $(t,x)\in Q_{\infty}$ with
$|x-x_0|+|t-t_0|<\delta$, we have
\begin{equation}\label{ineq.visc}
|\psi(t,x)-\psi(t_0,x_0)-\partial_t\psi(t_0,x_0)(t-t_0)|\le
\xi(|x-x_0|)+\omega(|t-t_0|).
\end{equation}

%%%%%%%%%%%%%%%%%%%%%%%%%%%%%%%%%%%%%%%%
\begin{definition}\label{def:vs}
An upper semicontinuous function $u:Q_{\infty}\to\real$ is a
viscosity subsolution to \eqref{FDE} in $Q_{\infty}$ if, whenever
$\psi\in\ca$ and $(t_0,x_0)\in Q_{\infty}$ are such that
\begin{equation*}
u(t_0,x_0)=\psi(t_0,x_0), \quad u(t,x)<\psi(t,x), \ \mbox{for all}\
(t,x)\in Q_{\infty}\setminus\{(t_0,x_0)\},
\end{equation*}
then
\begin{equation}
\left\{\begin{array}{ll}\partial_t\psi(t_0,x_0)\leq\Delta_{p}\psi(t_0,x_0)-|\nabla\psi(t_0,x_0)|^{q}
&
\ \hbox{if} \ \nabla\psi(t_0,x_0)\neq0,\\
\partial_t \psi(t_0,x_0)\leq 0 & \ \hbox{if} \
\nabla\psi(t_0,x_0)=0.\end{array}\right.
\end{equation}
A lower semicontinuous function $u:Q_{\infty}\to\real$ is a
viscosity supersolution to \eqref{FDE} in $Q_{\infty}$ if $-u$ is a
viscosity subsolution to \eqref{FDE} in $Q_{\infty}$. A continuous
function $u:Q_{\infty}\to\real$ is a viscosity solution to
\eqref{FDE} in $Q_{\infty}$ if it is a viscosity subsolution and
supersolution.
\end{definition}
%%%%%%%%%%%%%%%%%%%%%%%%%%%%%%%%%%%%%%%%

%%%%%%%%%%%%%%%%%%%%%%%%%%%%%%%%%%%%%%%%
%%%%%%%%%%%%%%%%%%%%%%%%%%%%%%%%%%%%%%%%
\bibliographystyle{plain}

\begin{thebibliography}{}

\end{thebibliography}


\begin{thebibliography}{1}

\bibitem{BL01}
S. Benachour and Ph. Lauren\ced{c}ot, \emph{Very singular solutions to
a nonlinear parabolic equation with absorption. I. Existence}, Proc.
Roy. Soc. Edinburgh, \textbf{131A} (2001), 27--44.

\bibitem{BKaL04} 
S.~Benachour, G.~Karch, and Ph.~Lauren\ced{c}ot, 
\emph{Asymptotic profiles of solutions to viscous Hamilton-Jacobi equations}, J. Math. Pures Appl., \textbf{83} (2004), 1275--1308.

\bibitem{BKL04}
S. Benachour, H. Koch, and Ph. Lauren\ced{c}ot, \emph{Very singular
solutions to a nonlinear parabolic equation with absorption. II.
Uniqueness}, Proc. Roy. Soc. Edinburgh, \textbf{134A} (2004),
39--54.

\bibitem{BNV}
H. Berestycki, L. Nirenberg, and S. R. S. Varadhan, \emph{The principal
eigenvalue and maximum principle for second order elliptic operators
in general domains}, Comm. Pure Appl. Math, \textbf{47} (1994), 47--92.

\bibitem{BPT}
H. Brezis, L.A. Peletier, and D. Terman, \emph{A very singular solution of
the heat equation with absorption}, Arch. Rational Mech. Anal.
\textbf{96} (1986), 185--209.

\bibitem{CQW03}
X. Chen, Y. Qi, and M. Wang, \emph{Self-similar singular solutions of a
$p$-Laplacian evolution equation with absorption}, J. Differential
Equations, \textbf{190} (2003), 1--15.

\bibitem{CIL92}
M. G.~Crandall, H.~Ishii, and P.-L.~Lions, \emph{User's guide to
viscosity solutions of second order partial differential equations},
Bull. Amer. Math. Soc. (N.S.), \textbf{27} (1992), 1--67.

\bibitem{EK88}
M.~Escobedo and O.~Kavian, \emph{Asymptotic behaviour of positive solutions of a nonlinear heat equation}, 
Houston J. Math., \textbf{14} (1988), 39--50. 

\bibitem{GT01}
D.~Gilbarg and N.S.~Trudinger, \textit{Elliptic partial differential equations of second order. Reprint of the 1998 ed.}, Classics in Mathematics, Springer, Berlin, 2001.

\bibitem{IL11}
R. Iagar and Ph. Lauren\ced{c}ot, \emph{Positivity, decay and
extinction for a singular diffusion equation with gradient
absorption}, submitted, preprint ArXiv no. 1104.1513, 2011.

\bibitem{IS}
H. Ishii and P. E. Souganidis, \emph{Generalized motion of
noncompact hypersurfaces with velocity having arbitrary growth on
the curvature tensor}, Tohoku Math. J., \textbf{47} (1995),
227--250.

\bibitem{KPV89}
S. Kamin, L.A. Peletier, and J. L. V\'azquez, \emph{Classification of
singular solutions of a nonlinear heat equation}, Duke Math. J.,
\textbf{58} (1989), 601--615.

\bibitem{KV92}
S. Kamin and J. L. V\'azquez, \emph{Singular solutions of some nonlinear
parabolic equations}, J. Analyse Math., \textbf{59} (1992), 51--74.

\bibitem{KVe88} S.~Kamin and L.~V\'eron, \emph{Existence and uniqueness
of the very singular solution of the porous media equation with absorption},
J. Analyse Math., \textbf{51} (1988), 245--258.

\bibitem{Le96} G.~Leoni, \emph{A very singular solution for the porous media
equation $u_t = \Delta\left({u^m}\right) - u^p$ when $0<m<1$},
J. Differential Equations, \textbf{132} (1996), 353--376.

\bibitem{Le97} G.~Leoni, \emph{On very singular self-similar solutions for
the porous media equation with absorption},
Differential Integral Equations, \textbf{10} (1997), 1123--1140.

\bibitem{OS}
M. Ohnuma and K. Sato, \emph{Singular degenerate parabolic equations
with applications to the $p$-Laplace diffusion equation}, Comm.
Partial Differential Equations, \textbf{22} (1997), 381--411.

\bibitem{PT86}
L.A.~Peletier and D. Terman, \emph{A very singular solution of the porous
media equation with absorption}, J. Differential Equations,
\textbf{65} (1986), 396--410.

\bibitem{PW88} L.A.~Peletier and J.~Wang,
\emph{A very singular solution of a quasilinear degenerate diffusion
equation with absorption}, Trans. Amer. Math. Soc., \textbf{307}
(1988), 813--826.

\bibitem{QW01}
Y. Qi and M. Wang, \emph{The self-similar profiles of generalized KPZ
equation}, Pacific J. Math., \textbf{201} (2001), 223--240.

\bibitem{Pe041}
P.~Shi, \emph{Self-similar singular solution of a $p$-Laplacian
evolution equation with gradient absorption term}, J. Partial
Differential Equations \textbf{17} (2004), 369--383.

\bibitem{Pe042}
P.~Shi, \emph{Self-similar very singular solution of a
$p$-Laplacian equation with gradient absorption: existence and
uniqueness}, J. Southeast University, \textbf{20} (2004), 381--386.

\bibitem{Va93} 
J.L.~V\'azquez,\emph{ Asymptotic behaviour of nonlinear parabolic equations. Anomalous exponents}, in: W.M.~Ni, 
L.A.~Peletier, J.L.~V\'azquez (Eds.), Degenerate Diffusions, IMA Vol. Math. Appl. \textbf{47}, Springer- 
Verlag, New York, 1993, pp.~215--228.

\end{thebibliography}

\end{document}